\documentclass[12pt,a4paper,reqno,dvipsnames]{amsart}

\usepackage[applemac]{inputenc}

\usepackage{pstricks}
\usepackage{pst-bezier}
\usepackage{graphicx}
\usepackage{color}
\usepackage{multido}
\usepackage{pst-plot}
\usepackage{pst-math}
\usepackage{pstricks-add}

\makeatletter
\def\psthyperbola{\pst@object{psthyperbola}}
\def\psthyperbola@i#1#2{%
  \pst@killglue
  \begingroup
  \use@par
  \psthyperbola@ii{#1}{#2}
  \psthyperbola@ii{#1 neg}{#2}
  \endgroup
}
\def\psthyperbola@ii#1#2{
  \addto@pscode{%
    /a {#1} bind def
    /b #2 def
    /d {1 t dup mul sub} bind def }
  \parametricplot{-.99}{.99}{%
    a t dup mul 1 add d div mul
    b t 2 mul d div mul}
}
\makeatother

\newrgbcolor{slategray4}{0.423529 0.482353 0.545098}
\newrgbcolor{darkseagreen1}{0.756863 1.0 0.756863}
\newrgbcolor{peachpuff}{1.0 0.854902 0.72549}
\newrgbcolor{lemonchiffon}{1.0 0.980392 0.803922}
\newrgbcolor{honeydew}{0.941176 1.0 0.941176}
\newrgbcolor{hellrot}{0.737255 0.560784 0.560784}
\newrgbcolor{sandybrown}{0.956863 0.643137 0.376471}
\newrgbcolor{aquamarine}{0.498039 1 0.831373}
\newrgbcolor{khakifour}{0.545098 0.52549 0.305882}
\newrgbcolor{lightcoral}{0.941176 0.501961 0.501961}
\newrgbcolor{palegreen}{0.560784 0.737255 0.560784}
\newrgbcolor{grey88}{0.878431 0.878431 0.878431}
\newrgbcolor{royalblue}{0.254902 0.411765 0.882353}
\newrgbcolor{lightyellow}{1 1 0.878431}
\newrgbcolor{mfzarthimmelblau}{0.6431372549019608 0.8274509803921568 0.9333333333333333}
\newrgbcolor{mfzartstahlblau}{0.792156862745098 0.8823529411764706 1.0}
\newrgbcolor{mfzarttuerkis}{0.7333333333333333 1.0 1.0}
\newrgbcolor{mfzartazur}{0.8784313725490196 0.9333333333333333 0.9333333333333333}
\newrgbcolor{mfgedeckttuerkis}{0.25098039215686274 0.8784313725490196 0.8156862745098039}
\newrgbcolor{mfzartbraun}{1.0 0.8274509803921568 0.6078431372549019}
\newrgbcolor{mfolivgruen}{0.6352941176470588 0.803921568627451 0.35294117647058826}
\newrgbcolor{mfzartgruen}{0.7568627450980392 1.0 0.7568627450980392}
\newrgbcolor{mfkhaki}{1.0 0.9647058823529412 0.5607843137254902}
\newrgbcolor{mfzartorange}{1.0 0.8941176470588236 0.7686274509803922}
\newrgbcolor{mfzartpfirsich}{0.9333333333333333 0.8352941176470589 0.7176470588235294}
\newrgbcolor{mfhonigtau}{0.9411764705882353 1.0 0.9411764705882353}
\newrgbcolor{mfzartrosa}{1.0 0.8941176470588236 0.8823529411764706}
\newrgbcolor{mfgedeckttomate}{0.803921568627451 0.30980392156862746 0.2235294117647059}
\newrgbcolor{mflavendel}{0.9019607843137255 0.9019607843137255 0.9803921568627451}
\newrgbcolor{mfgedecktlila}{0.803921568627451 0.4117647058823529 0.788235294117647}
\newrgbcolor{mfgedecktzwetschge}{0.803921568627451 0.5882352941176471 0.803921568627451}
\newrgbcolor{mfzartweizen}{0.9333333333333333 0.8470588235294118 0.6823529411764706}
\newrgbcolor{mfzartzitrone}{1.0 0.9803921568627451 0.803921568627451}
\newrgbcolor{mfzartgold}{0.9333333333333333 0.9098039215686274 0.6666666666666666}
\newrgbcolor{mfzartgruengelb}{0.803921568627451 0.803921568627451 0.0}

\newgray{backgroundgray}{0.97}
\newgray{verylightgray}{0.9}
\newgray{ratherlightgray}{0.8}

\def\bit{\begin{itemize}}
\def\eit{\end{itemize}}
\def\beq#1{\begin{equation}\label{#1}}
\def\eeq{\end{equation}}
\def\bpf{\begin{proof}\noindent}
\def\epf{\end{proof}}
\def\ben{\begin{enumerate}}
\def\een{\end{enumerate}}
\def\bmat{\begin{bmatrix}}
\def\emat{\end{bmatrix}}

\def\EM#1{{\em #1}}

\def\heute{\number\day.~\ifcase\month\or
 J\"anner\or Februar\or M\"arz\or April\or Mai\or Juni\or
 Juli\or August\or September\or Oktober\or November\or Dezember\fi
 \space\number\year}

\def\qpoly#1{{\scriptstyle{\mathbf Q}{\of{#1}}}}
\def\bisp#1{{\scriptstyle{\mathbf P_p}{\of{#1}}}}
\def\pbisa#1{{\scriptstyle{\mathbf P_a}{\of{#1}}}}
\def\bisd#1{{\scriptstyle{\mathbf P_d}{\of{#1}}}}
\def\pfactor#1{{\scriptstyle{\mathbf P}{\of{#1}}}}

\def\evencount#1{{\scriptstyle{\mathbf e}{\of{#1}}}}
\def\oddcount#1{{\scriptstyle{\mathbf o}{\of{#1}}}}

\def\poly#1#2{{\scriptstyle{\mathbf #1}{\of{#2}}}}
\def\macmahon#1{{\scriptstyle{\mathbf{mm}\of{#1}}}}
\def\pochhammer#1#2{\pas{#1}_{#2}}

\def\generalfactor#1{\poly{F}{#1}}
\def\specialfactor#1{\poly{F^\prime}{#1}}

\title{Tilings of damaged hexagons}
\author{Markus Fulmek}

\def\of#1{\left(#1\right)} 
\def\pas#1{\left(#1\right)} 
\def\brk#1{\left[#1\right]} 

\def\defeq{\stackrel{\text{\tiny def}}{=}}

\def\N{{\mathbb N}}

\def\Z{{\mathbb Z}}

\def\0{{\mathbf 0}} 
\def\1{{\mathbf 1}} 

\def\Iverson#1{\left[#1\right]}

\def\initvertex#1{\alpha\of{e}}
\def\xxxinitvertex#1{\alpha\of{e}}
\def\termvertex#1{\omega\of{e}}
\def\xxxtermvertex#1{\omega\of{e}}

\def\floor#1{\left\lfloor #1\right\rfloor}

\def\ceil#1{\left\lceil #1\right\rceil}

\def\CT1{CT1}
\def\xxxCT1{CT1}

\let\phi\varphi
\let\cps\centerpspicture

\long\def\psfigurescommented#1#2#3#4{\begin{figure}%
\begin{center}%
#1  %
\end{center}%
\vskip1em%
\parbox[top]{0.9\textwidth}{{\footnotesize\gray #3}} %
\caption{#2}\label{fig:#4}%
\end{figure}}

\def\nilp{non\-in\-ter\-sec\-ting lat\-ti\-ce path}
\def\eye{\mathbf 1}
\def\nought{\mathbf 0}

\def\secA#1{\section{#1}}
\def\secB#1{\subsection{#1}}
\def\secC#1{\subsubsection{#1}}

\newtheorem{lem}{Lemma}

\newtheorem{pro}{Proposition}
\newtheorem{rem}{Remark}
\newtheorem{cor}{Corollary}
\newtheorem{con}{Conjecture}
\newtheorem{ex}{Example}
\def\figref#1{Fi\-gu\-re~\ref{fig:#1}}
\def\proref#1{Pro\-po\-si\-ti\-on~\ref{pro:#1}}

\begin{document}
\bibliographystyle{plain}
\maketitle

\tableofcontents

\def\cps#1{\begin{center}\input graphics/#1 \end{center}}

\secA{Introduction}

In a recent paper \cite{Byun:2022:LTOAHWAHI}, Byun presented nice formulas for
the enumeration of lozenge tilings of certain hexagonal regions with ``intrusions''.

This paper attempts to generalise some of Byun's investigations.
It is organised as follows:
\bit
\item In section \ref{sec:damaged-hexagons}, we present the background needed
	for the considerations
	in this paper. This material is well--known to the expert, and the non--expert can easily
	conceive it from illustrations: Hence, in most cases
	we shall avoid lengthy formal definitions and present illustrative pictures instead.
\item In section \ref{sec:determinants}, we explain the bijection between lozenge tilings
	and \nilp s and recall
	\bit
	\item the Lindstr\"om--Gessel--Viennot method for counting \nilp s
	\item and Dodgson's condensation formula.
	\eit
\item In section \ref{sec:even-intrusions}, we apply these considerations to a generalisation
	of Byun's investigations, present solutions to (simple) special cases and
	formulate a conjecture for the general case. Moreover, by straightforward matrix
	manipulation we rewrite a simple special case as a summation formula. 
\eit

Parts of the considerations involve lengthy manipulations of rational functions and polynomials:
The software \EM{Mathematica} and the Python--library \EM{sympy} was used to help with
such manipulations. Moreover, Zeilberger's algorithm \cite{Zeilberger:1991:TMOCT}
(which not only gives the
result, but also an ``automated proof'') was employed; in the implementation of Paule, Schorn and Riese \cite{zb:risc} . 

\secA{Damaged hexagons and Byun's formulas}
\label{sec:damaged-hexagons}
\secB{Hexagons with intrusions in the triangular lattice}
\psfigurescommented{
\psset{unit=0.6cm}
\begin{pspicture}(-1.95,-0.6)(6.95,7.3782)
\pspolygon[linecolor=white,fillstyle=solid,fillcolor=backgroundgray,linearc=0.3](-1.95,-0.6)(6.95,-0.6)(6.95,7.3782)(-1.95,7.3782)

\psset{linewidth=0.5pt,linecolor=gray,linestyle=solid,fillstyle=none}
\psline(-1.5,2.5981)(1,6.9282)
\psline(-1,1.7321)(2,6.9282)
\psline(-0.5,0.86603)(3,6.9282)
\psline(0,0)(4,6.9282)
\psline(1,0)(5,6.9282)
\psline(2,0)(5.5,6.0622)
\psline(3,0)(6,5.1962)
\psline(4,0)(6.5,4.3301)
\psline(0,0)(-1.5,2.5981)
\psline(1,0)(-1,3.4641)
\psline(2,0)(-0.5,4.3301)
\psline(3,0)(0,5.1962)
\psline(4,0)(0.5,6.0622)
\psline(4.5,0.86603)(1,6.9282)
\psline(5,1.7321)(2,6.9282)
\psline(5.5,2.5981)(3,6.9282)
\psline(6,3.4641)(4,6.9282)
\psline(6.5,4.3301)(5,6.9282)
\psline(4,0)(0,0)
\psline(4.5,0.86603)(-0.5,0.86603)
\psline(5,1.7321)(-1,1.7321)
\psline(5.5,2.5981)(-1.5,2.5981)
\psline(6,3.4641)(-1,3.4641)
\psline(6.5,4.3301)(-0.5,4.3301)
\psline(6,5.1962)(0,5.1962)
\psline(5.5,6.0622)(0.5,6.0622)
\psline(5,6.9282)(1,6.9282)
\psset{fillstyle=solid,fillcolor=lightgray,linecolor=lightgray}
\pspolygon(2,0)(2.5,0.86603)(1.5,0.86603)
\pspolygon(1.5,0.86603)(2.5,0.86603)(2,1.7321)
\pspolygon(2,1.7321)(2.5,2.5981)(1.5,2.5981)
\pspolygon(1.5,2.5981)(2.5,2.5981)(2,3.4641)
\psset{linewidth=1pt,linecolor=black,linestyle=solid,fillstyle=none}
\pspolygon[fillstyle=none](0,0)(4,0)(6.5,4.3301)(5,6.9282)(1,6.9282)(-1.5,2.5981)
\uput[-90.0](2,0){\tiny $a=4$}
\uput[-30.0](5.25,2.1651){\tiny $b=5$}
\uput[30.0](5.75,5.6292){\tiny $c=3$}
\uput[90.0](3,6.9282){\tiny $a=4$}
\uput[150.0](-0.25,4.7631){\tiny $b=5$}
\uput[-150.0](-0.75,1.299){\tiny $c=3$}
\psline[linestyle=solid,linewidth=0.05,linecolor=red](0,0)(0,-0.15)
\rput(-0.075,0.4){ {\tiny\red $4$}}
\psline[linestyle=solid,linewidth=0.05,linecolor=red](1,0)(1,-0.15)
\rput(0.925,0.4){ {\tiny\red $3$}}
\psline[linestyle=solid,linewidth=0.05,linecolor=red](2,0)(2,-0.15)
\rput(1.925,0.4){ {\tiny\red $2$}}
\psline[linestyle=solid,linewidth=0.05,linecolor=red](3,0)(3,-0.15)
\rput(2.925,0.4){ {\tiny\red $1$}}
\psline[linestyle=solid,linewidth=0.05,linecolor=red](4,0)(4,-0.15)
\rput(3.925,0.4){ {\tiny\red $0$}}
\end{pspicture} 
\psset{unit=0.6cm}
\begin{pspicture}(-1.95,-0.45)(6.95,7.3782)
\pspolygon[linecolor=white,fillstyle=solid,fillcolor=backgroundgray,linearc=0.3](-1.95,-0.45)(6.95,-0.45)(6.95,7.3782)(-1.95,7.3782)

\psset{linewidth=0.5pt,linecolor=gray,linestyle=solid,fillstyle=none}
\pspolygon[fillstyle=solid,fillcolor=Apricot,linecolor=white](0,0)(4,0)(6.5,4.3301)(5,6.9282)(1,6.9282)(-1.5,2.5981)
\psline[linecolor=white](-1,1.7321)(2,6.9282)
\psline[linecolor=white](-0.5,0.86603)(3,6.9282)
\psline[linecolor=white](0,0)(4,6.9282)
\psline[linecolor=white](1,0)(5,6.9282)
\psline[linecolor=white](2,0)(5.5,6.0622)
\psline[linecolor=white](3,0)(6,5.1962)
\psline[linecolor=white](-1,3.4641)(1,0)
\psline[linecolor=white](-0.5,4.3301)(2,0)
\psline[linecolor=white](0,5.1962)(3,0)
\psline[linecolor=white](0.5,6.0622)(4,0)
\psline[linecolor=white](1,6.9282)(4.5,0.86603)
\psline[linecolor=white](2,6.9282)(5,1.7321)
\psline[linecolor=white](3,6.9282)(5.5,2.5981)
\psline[linecolor=white](4,6.9282)(6,3.4641)
\psset{fillstyle=solid,fillcolor=white,linecolor=white}
\pspolygon(2,0)(2.5,0.86603)(1.5,0.86603)
\pspolygon(1.5,0.86603)(2.5,0.86603)(2,1.7321)
\pspolygon(2,1.7321)(2.5,2.5981)(1.5,2.5981)
\pspolygon(1.5,2.5981)(2.5,2.5981)(2,3.4641)
\pspolygon[fillstyle=solid,fillcolor=Tan,linecolor=white](3,0)(4,0)(4.5,0.86603)(3.5,0.86603)
\pspolygon[fillstyle=solid,fillcolor=Tan,linecolor=white](3.5,0.86603)(4.5,0.86603)(5,1.7321)(4,1.7321)
\pspolygon[fillstyle=solid,fillcolor=Mahogany,linecolor=white](4,1.7321)(5,1.7321)(4.5,2.5981)(3.5,2.5981)
\pspolygon[fillstyle=solid,fillcolor=Tan,linecolor=white](3.5,2.5981)(4.5,2.5981)(5,3.4641)(4,3.4641)
\pspolygon[fillstyle=solid,fillcolor=Tan,linecolor=white](4,3.4641)(5,3.4641)(5.5,4.3301)(4.5,4.3301)
\pspolygon[fillstyle=solid,fillcolor=Mahogany,linecolor=white](4.5,4.3301)(5.5,4.3301)(5,5.1962)(4,5.1962)
\pspolygon[fillstyle=solid,fillcolor=Tan,linecolor=white](4,5.1962)(5,5.1962)(5.5,6.0622)(4.5,6.0622)
\pspolygon[fillstyle=solid,fillcolor=Mahogany,linecolor=white](4.5,6.0622)(5.5,6.0622)(5,6.9282)(4,6.9282)
\pspolygon[fillstyle=solid,fillcolor=Tan,linecolor=white](2,0)(3,0)(3.5,0.86603)(2.5,0.86603)
\pspolygon[fillstyle=solid,fillcolor=Mahogany,linecolor=white](2.5,0.86603)(3.5,0.86603)(3,1.7321)(2,1.7321)
\pspolygon[fillstyle=solid,fillcolor=Tan,linecolor=white](2,1.7321)(3,1.7321)(3.5,2.5981)(2.5,2.5981)
\pspolygon[fillstyle=solid,fillcolor=Tan,linecolor=white](2.5,2.5981)(3.5,2.5981)(4,3.4641)(3,3.4641)
\pspolygon[fillstyle=solid,fillcolor=Mahogany,linecolor=white](3,3.4641)(4,3.4641)(3.5,4.3301)(2.5,4.3301)
\pspolygon[fillstyle=solid,fillcolor=Tan,linecolor=white](2.5,4.3301)(3.5,4.3301)(4,5.1962)(3,5.1962)
\pspolygon[fillstyle=solid,fillcolor=Tan,linecolor=white](3,5.1962)(4,5.1962)(4.5,6.0622)(3.5,6.0622)
\pspolygon[fillstyle=solid,fillcolor=Mahogany,linecolor=white](3.5,6.0622)(4.5,6.0622)(4,6.9282)(3,6.9282)
\pspolygon[fillstyle=solid,fillcolor=Mahogany,linecolor=white](1,0)(2,0)(1.5,0.86603)(0.5,0.86603)
\pspolygon[fillstyle=solid,fillcolor=Tan,linecolor=white](0.5,0.86603)(1.5,0.86603)(2,1.7321)(1,1.7321)
\pspolygon[fillstyle=solid,fillcolor=Mahogany,linecolor=white](1,1.7321)(2,1.7321)(1.5,2.5981)(0.5,2.5981)
\pspolygon[fillstyle=solid,fillcolor=Tan,linecolor=white](0.5,2.5981)(1.5,2.5981)(2,3.4641)(1,3.4641)
\pspolygon[fillstyle=solid,fillcolor=Tan,linecolor=white](1,3.4641)(2,3.4641)(2.5,4.3301)(1.5,4.3301)
\pspolygon[fillstyle=solid,fillcolor=Mahogany,linecolor=white](1.5,4.3301)(2.5,4.3301)(2,5.1962)(1,5.1962)
\pspolygon[fillstyle=solid,fillcolor=Tan,linecolor=white](1,5.1962)(2,5.1962)(2.5,6.0622)(1.5,6.0622)
\pspolygon[fillstyle=solid,fillcolor=Tan,linecolor=white](1.5,6.0622)(2.5,6.0622)(3,6.9282)(2,6.9282)
\pspolygon[fillstyle=solid,fillcolor=Mahogany,linecolor=white](0,0)(1,0)(0.5,0.86603)(-0.5,0.86603)
\pspolygon[fillstyle=solid,fillcolor=Mahogany,linecolor=white](-0.5,0.86603)(0.5,0.86603)(0,1.7321)(-1,1.7321)
\pspolygon[fillstyle=solid,fillcolor=Tan,linecolor=white](-1,1.7321)(0,1.7321)(0.5,2.5981)(-0.5,2.5981)
\pspolygon[fillstyle=solid,fillcolor=Tan,linecolor=white](-0.5,2.5981)(0.5,2.5981)(1,3.4641)(0,3.4641)
\pspolygon[fillstyle=solid,fillcolor=Mahogany,linecolor=white](0,3.4641)(1,3.4641)(0.5,4.3301)(-0.5,4.3301)
\pspolygon[fillstyle=solid,fillcolor=Tan,linecolor=white](-0.5,4.3301)(0.5,4.3301)(1,5.1962)(0,5.1962)
\pspolygon[fillstyle=solid,fillcolor=Tan,linecolor=white](0,5.1962)(1,5.1962)(1.5,6.0622)(0.5,6.0622)
\pspolygon[fillstyle=solid,fillcolor=Tan,linecolor=white](0.5,6.0622)(1.5,6.0622)(2,6.9282)(1,6.9282)
\psset{linewidth=1pt,linecolor=black,linestyle=solid,fillstyle=none}
\pspolygon[fillstyle=none](0,0)(4,0)(6.5,4.3301)(5,6.9282)(1,6.9282)(-1.5,2.5981)
\psset{linewidth=0.1,linecolor=blue,linearc=0.15}
\end{pspicture} 
\psset{unit=0.6cm}
\begin{pspicture}(-1.95,-0.575)(6.95,7.5032)
\pspolygon[linecolor=white,fillstyle=solid,fillcolor=backgroundgray,linearc=0.3](-1.95,-0.575)(6.95,-0.575)(6.95,7.5032)(-1.95,7.5032)

\psset{linewidth=0.5pt,linecolor=gray,linestyle=solid,fillstyle=none}
\pspolygon[fillstyle=solid,fillcolor=Apricot,linecolor=white](0,0)(4,0)(6.5,4.3301)(5,6.9282)(1,6.9282)(-1.5,2.5981)
\psline[linecolor=white](-1,1.7321)(2,6.9282)
\psline[linecolor=white](-0.5,0.86603)(3,6.9282)
\psline[linecolor=white](0,0)(4,6.9282)
\psline[linecolor=white](1,0)(5,6.9282)
\psline[linecolor=white](2,0)(5.5,6.0622)
\psline[linecolor=white](3,0)(6,5.1962)
\psline[linecolor=white](-1,3.4641)(1,0)
\psline[linecolor=white](-0.5,4.3301)(2,0)
\psline[linecolor=white](0,5.1962)(3,0)
\psline[linecolor=white](0.5,6.0622)(4,0)
\psline[linecolor=white](1,6.9282)(4.5,0.86603)
\psline[linecolor=white](2,6.9282)(5,1.7321)
\psline[linecolor=white](3,6.9282)(5.5,2.5981)
\psline[linecolor=white](4,6.9282)(6,3.4641)
\psset{fillstyle=solid,fillcolor=white,linecolor=white}
\pspolygon(2,0)(2.5,0.86603)(1.5,0.86603)
\pspolygon(1.5,0.86603)(2.5,0.86603)(2,1.7321)
\pspolygon(2,1.7321)(2.5,2.5981)(1.5,2.5981)
\pspolygon(1.5,2.5981)(2.5,2.5981)(2,3.4641)
\pspolygon[fillstyle=solid,fillcolor=Tan,linecolor=white](3,0)(4,0)(4.5,0.86603)(3.5,0.86603)
\pspolygon[fillstyle=solid,fillcolor=Tan,linecolor=white](3.5,0.86603)(4.5,0.86603)(5,1.7321)(4,1.7321)
\pspolygon[fillstyle=solid,fillcolor=Mahogany,linecolor=white](4,1.7321)(5,1.7321)(4.5,2.5981)(3.5,2.5981)
\pspolygon[fillstyle=solid,fillcolor=Tan,linecolor=white](3.5,2.5981)(4.5,2.5981)(5,3.4641)(4,3.4641)
\pspolygon[fillstyle=solid,fillcolor=Tan,linecolor=white](4,3.4641)(5,3.4641)(5.5,4.3301)(4.5,4.3301)
\pspolygon[fillstyle=solid,fillcolor=Mahogany,linecolor=white](4.5,4.3301)(5.5,4.3301)(5,5.1962)(4,5.1962)
\pspolygon[fillstyle=solid,fillcolor=Tan,linecolor=white](4,5.1962)(5,5.1962)(5.5,6.0622)(4.5,6.0622)
\pspolygon[fillstyle=solid,fillcolor=Mahogany,linecolor=white](4.5,6.0622)(5.5,6.0622)(5,6.9282)(4,6.9282)
\pspolygon[fillstyle=solid,fillcolor=Tan,linecolor=white](2,0)(3,0)(3.5,0.86603)(2.5,0.86603)
\pspolygon[fillstyle=solid,fillcolor=Mahogany,linecolor=white](2.5,0.86603)(3.5,0.86603)(3,1.7321)(2,1.7321)
\pspolygon[fillstyle=solid,fillcolor=Tan,linecolor=white](2,1.7321)(3,1.7321)(3.5,2.5981)(2.5,2.5981)
\pspolygon[fillstyle=solid,fillcolor=Tan,linecolor=white](2.5,2.5981)(3.5,2.5981)(4,3.4641)(3,3.4641)
\pspolygon[fillstyle=solid,fillcolor=Mahogany,linecolor=white](3,3.4641)(4,3.4641)(3.5,4.3301)(2.5,4.3301)
\pspolygon[fillstyle=solid,fillcolor=Tan,linecolor=white](2.5,4.3301)(3.5,4.3301)(4,5.1962)(3,5.1962)
\pspolygon[fillstyle=solid,fillcolor=Tan,linecolor=white](3,5.1962)(4,5.1962)(4.5,6.0622)(3.5,6.0622)
\pspolygon[fillstyle=solid,fillcolor=Mahogany,linecolor=white](3.5,6.0622)(4.5,6.0622)(4,6.9282)(3,6.9282)
\pspolygon[fillstyle=solid,fillcolor=Mahogany,linecolor=white](1,0)(2,0)(1.5,0.86603)(0.5,0.86603)
\pspolygon[fillstyle=solid,fillcolor=Tan,linecolor=white](0.5,0.86603)(1.5,0.86603)(2,1.7321)(1,1.7321)
\pspolygon[fillstyle=solid,fillcolor=Mahogany,linecolor=white](1,1.7321)(2,1.7321)(1.5,2.5981)(0.5,2.5981)
\pspolygon[fillstyle=solid,fillcolor=Tan,linecolor=white](0.5,2.5981)(1.5,2.5981)(2,3.4641)(1,3.4641)
\pspolygon[fillstyle=solid,fillcolor=Tan,linecolor=white](1,3.4641)(2,3.4641)(2.5,4.3301)(1.5,4.3301)
\pspolygon[fillstyle=solid,fillcolor=Mahogany,linecolor=white](1.5,4.3301)(2.5,4.3301)(2,5.1962)(1,5.1962)
\pspolygon[fillstyle=solid,fillcolor=Tan,linecolor=white](1,5.1962)(2,5.1962)(2.5,6.0622)(1.5,6.0622)
\pspolygon[fillstyle=solid,fillcolor=Tan,linecolor=white](1.5,6.0622)(2.5,6.0622)(3,6.9282)(2,6.9282)
\pspolygon[fillstyle=solid,fillcolor=Mahogany,linecolor=white](0,0)(1,0)(0.5,0.86603)(-0.5,0.86603)
\pspolygon[fillstyle=solid,fillcolor=Mahogany,linecolor=white](-0.5,0.86603)(0.5,0.86603)(0,1.7321)(-1,1.7321)
\pspolygon[fillstyle=solid,fillcolor=Tan,linecolor=white](-1,1.7321)(0,1.7321)(0.5,2.5981)(-0.5,2.5981)
\pspolygon[fillstyle=solid,fillcolor=Tan,linecolor=white](-0.5,2.5981)(0.5,2.5981)(1,3.4641)(0,3.4641)
\pspolygon[fillstyle=solid,fillcolor=Mahogany,linecolor=white](0,3.4641)(1,3.4641)(0.5,4.3301)(-0.5,4.3301)
\pspolygon[fillstyle=solid,fillcolor=Tan,linecolor=white](-0.5,4.3301)(0.5,4.3301)(1,5.1962)(0,5.1962)
\pspolygon[fillstyle=solid,fillcolor=Tan,linecolor=white](0,5.1962)(1,5.1962)(1.5,6.0622)(0.5,6.0622)
\pspolygon[fillstyle=solid,fillcolor=Tan,linecolor=white](0.5,6.0622)(1.5,6.0622)(2,6.9282)(1,6.9282)
\psset{linewidth=1pt,linecolor=black,linestyle=solid,fillstyle=none}
\pspolygon[fillstyle=none](0,0)(4,0)(6.5,4.3301)(5,6.9282)(1,6.9282)(-1.5,2.5981)
\psset{linewidth=0.1,linecolor=blue,linearc=0.15}
\psline(3.5,0)(4.5,1.7321)(4,2.5981)(5,4.3301)(4.5,5.1962)(5,6.0622)(4.5,6.9282)
\psline(2.5,0)(3,0.86603)(2.5,1.7321)(3.5,3.4641)(3,4.3301)(4,6.0622)(3.5,6.9282)
\psline(1.5,0)(1,0.86603)(1.5,1.7321)(1,2.5981)(2,4.3301)(1.5,5.1962)(2.5,6.9282)
\psline(0.5,0)(-0.5,1.7321)(0.5,3.4641)(0,4.3301)(1.5,6.9282)
\pscircle[fillstyle=solid,linecolor=black,linewidth=0.5pt,fillcolor=red,](3.5,0){0.125}
\pscircle[fillstyle=solid,linecolor=black,linewidth=0.5pt,fillcolor=red,](2.5,0){0.125}
\pscircle[fillstyle=solid,linecolor=black,linewidth=0.5pt,fillcolor=red,](1.5,0){0.125}
\pscircle[fillstyle=solid,linecolor=black,linewidth=0.5pt,fillcolor=red,](0.5,0){0.125}
\pscircle[fillstyle=solid,linecolor=black,linewidth=0.5pt,fillcolor=red,](2,0.86603){0.125}
\pscircle[fillstyle=solid,linecolor=black,linewidth=0.5pt,fillcolor=red,](2,2.5981){0.125}
\pscircle[fillstyle=solid,linecolor=black,linewidth=0.5pt,fillcolor=red,fillcolor=green](4.5,6.9282){0.125}
\pscircle[fillstyle=solid,linecolor=black,linewidth=0.5pt,fillcolor=red,fillcolor=green](3.5,6.9282){0.125}
\pscircle[fillstyle=solid,linecolor=black,linewidth=0.5pt,fillcolor=red,fillcolor=green](2.5,6.9282){0.125}
\pscircle[fillstyle=solid,linecolor=black,linewidth=0.5pt,fillcolor=red,fillcolor=green](1.5,6.9282){0.125}
\pswedge[fillstyle=solid,linecolor=black,linewidth=0.5pt,fillcolor=green](2,0.86603){0.125}{0}{180}
\pswedge[fillstyle=solid,linecolor=black,linewidth=0.5pt,fillcolor=green](2,2.5981){0.125}{0}{180}
\end{pspicture} 
\psset{unit=0.6cm}
\begin{pspicture}(-3.95,-0.95)(5.95,6.95)
\pspolygon[linecolor=white,fillstyle=solid,fillcolor=backgroundgray,linearc=0.3](-3.95,-0.95)(5.95,-0.95)(5.95,6.95)(-3.95,6.95)

\psset{linewidth=0.5pt,linecolor=gray,linestyle=solid,fillstyle=none}
\psline(-3,-0.5)(-3,6.5)
\psline(-2,-0.5)(-2,6.5)
\psline(-1,-0.5)(-1,6.5)
\psline(0,-0.5)(0,6.5)
\psline(1,-0.5)(1,6.5)
\psline(2,-0.5)(2,6.5)
\psline(3,-0.5)(3,6.5)
\psline(4,-0.5)(4,6.5)
\psline(5,-0.5)(5,6.5)
\psline(-3.5,1)(5.5,1)
\psline(-3.5,2)(5.5,2)
\psline(-3.5,3)(5.5,3)
\psline(-3.5,4)(5.5,4)
\psline(-3.5,5)(5.5,5)
\psline(-3.5,6)(5.5,6)
\psset{linewidth=1pt,linecolor=black,linestyle=solid,fillstyle=none}
\psline{->}(-3.5,0)(5.5,0)
\psline{->}(0,-0.5)(0,6.5)
\psset{linewidth=0.1,linecolor=blue,linearc=0.1}
\psline(0,0)(2,0)(2,1)(4,1)(4,2)(5,2)(5,3)
\psline(-1,1)(0,1)(0,2)(2,2)(2,3)(4,3)(4,4)
\psline(-2,2)(-2,3)(-1,3)(-1,4)(1,4)(1,5)(3,5)
\psline(-3,3)(-3,5)(-1,5)(-1,6)(2,6)
\pscircle[fillstyle=solid,linecolor=black,linewidth=0.5pt,fillcolor=red,](0,0){0.125}
\pscircle[fillstyle=solid,linecolor=black,linewidth=0.5pt,fillcolor=red,](-1,1){0.125}
\pscircle[fillstyle=solid,linecolor=black,linewidth=0.5pt,fillcolor=red,](-2,2){0.125}
\pscircle[fillstyle=solid,linecolor=black,linewidth=0.5pt,fillcolor=red,](-3,3){0.125}
\pscircle[fillstyle=solid,linecolor=black,linewidth=0.5pt,fillcolor=red,](-1,2){0.125}
\pscircle[fillstyle=solid,linecolor=black,linewidth=0.5pt,fillcolor=red,](0,3){0.125}
\pscircle[fillstyle=solid,linecolor=black,linewidth=0.5pt,fillcolor=red,fillcolor=green](5,3){0.125}
\pscircle[fillstyle=solid,linecolor=black,linewidth=0.5pt,fillcolor=red,fillcolor=green](4,4){0.125}
\pscircle[fillstyle=solid,linecolor=black,linewidth=0.5pt,fillcolor=red,fillcolor=green](3,5){0.125}
\pscircle[fillstyle=solid,linecolor=black,linewidth=0.5pt,fillcolor=red,fillcolor=green](2,6){0.125}
\pswedge[fillstyle=solid,linecolor=black,linewidth=0.5pt,fillcolor=green](-1,2){0.125}{-45}{135}
\pswedge[fillstyle=solid,linecolor=black,linewidth=0.5pt,fillcolor=green](0,3){0.125}{-45}{135}
\end{pspicture} %
}{
The hexagon with side lengths $\pas{a,b,c}=\pas{4,5,3}$ and \EM{even intrusion} of length $d=2$ at position $p=2$.%
}{
The upper left picture shows the hexagon with side lengths $\pas{a,b,c}=\pas{4,5,3}$ in the triangular
lattice with an \EM{even intrusion} of length $d=2$ (marked as gray triangles) at position $p=2$ (possible positions of
intrusions are indicated by ticks at the base line of the hexagon). The upper right
picture shows a \EM{lozenge tiling} of this hexagonal region, and the lower left picture
shows the \EM{same} tiling together with the corresponding family of \EM{nonintersecting lattice paths}
(starting points of the paths are coloured red, and ending points are coloured green; points which
are starting \EM{and} ending points --- these correspond to the ``intrusion'' --- are coloured red \EM{and} green):
It is a well--known fact that this correspondence is a \EM{bijection}.

The lower right picture shows the family of nonintersecting lattice paths in the
integer lattice $\Z\times\Z$: These are obtained by tilting the paths shown in the
picture to the left and shifting them in the plane such that the lowest starting
point coincides with the origin $\pas{0,0}$. Altogether, this gives $a=4$ \EM{lateral}
starting points plus $d=2$ \EM{intrusive} starting points
$$
\pas{
	\underbrace{(0,0),(-1,1),(-2,2),(-3,3)}_{\text{lateral}},
	\underbrace{(-1,2),(0,3)}_{\text{intrusive}}
},
$$
and $a=4$ \EM{lateral} ending points plus $d=2$ \EM{intrusive} ending points
$$\pas{
	\underbrace{(5,3),(4,4),(3,5),(2,6)}_{\text{lateral}},
	\underbrace{(-1,2),(0,3)}_{\text{intrusive}}
}.
$$

}{
hexc%
}

\psfigurescommented{
\psset{unit=0.6cm}
\begin{pspicture}(-1.95,-0.6)(6.95,7.3782)
\pspolygon[linecolor=white,fillstyle=solid,fillcolor=backgroundgray,linearc=0.3](-1.95,-0.6)(6.95,-0.6)(6.95,7.3782)(-1.95,7.3782)

\psset{linewidth=0.5pt,linecolor=gray,linestyle=solid,fillstyle=none}
\psline(-1.5,2.5981)(1,6.9282)
\psline(-1,1.7321)(2,6.9282)
\psline(-0.5,0.86603)(3,6.9282)
\psline(0,0)(4,6.9282)
\psline(1,0)(5,6.9282)
\psline(2,0)(5.5,6.0622)
\psline(3,0)(6,5.1962)
\psline(4,0)(6.5,4.3301)
\psline(0,0)(-1.5,2.5981)
\psline(1,0)(-1,3.4641)
\psline(2,0)(-0.5,4.3301)
\psline(3,0)(0,5.1962)
\psline(4,0)(0.5,6.0622)
\psline(4.5,0.86603)(1,6.9282)
\psline(5,1.7321)(2,6.9282)
\psline(5.5,2.5981)(3,6.9282)
\psline(6,3.4641)(4,6.9282)
\psline(6.5,4.3301)(5,6.9282)
\psline(4,0)(0,0)
\psline(4.5,0.86603)(-0.5,0.86603)
\psline(5,1.7321)(-1,1.7321)
\psline(5.5,2.5981)(-1.5,2.5981)
\psline(6,3.4641)(-1,3.4641)
\psline(6.5,4.3301)(-0.5,4.3301)
\psline(6,5.1962)(0,5.1962)
\psline(5.5,6.0622)(0.5,6.0622)
\psline(5,6.9282)(1,6.9282)
\psset{fillstyle=solid,fillcolor=lightgray,linecolor=lightgray}
\pspolygon(2,0)(3,0)(2.5,0.86603)
\pspolygon(2.5,0.86603)(3,1.7321)(2,1.7321)
\pspolygon(2,1.7321)(3,1.7321)(2.5,2.5981)
\pspolygon(2.5,2.5981)(3,3.4641)(2,3.4641)
\pspolygon(2,3.4641)(3,3.4641)(2.5,4.3301)
\pspolygon(2.5,4.3301)(3,5.1962)(2,5.1962)
\psset{linewidth=1pt,linecolor=black,linestyle=solid,fillstyle=none}
\pspolygon[fillstyle=none](0,0)(4,0)(6.5,4.3301)(5,6.9282)(1,6.9282)(-1.5,2.5981)
\uput[-90.0](2,0){\tiny $a=4$}
\uput[-30.0](5.25,2.1651){\tiny $b=5$}
\uput[30.0](5.75,5.6292){\tiny $c=3$}
\uput[90.0](3,6.9282){\tiny $a=4$}
\uput[150.0](-0.25,4.7631){\tiny $b=5$}
\uput[-150.0](-0.75,1.299){\tiny $c=3$}
\psline[linestyle=solid,linewidth=0.05,linecolor=red](0.5,0)(0.5,-0.15)
\rput(0.425,0.4){ {\tiny\red $4$}}
\psline[linestyle=solid,linewidth=0.05,linecolor=red](1.5,0)(1.5,-0.15)
\rput(1.425,0.4){ {\tiny\red $3$}}
\psline[linestyle=solid,linewidth=0.05,linecolor=red](2.5,0)(2.5,-0.15)
\rput(2.425,0.4){ {\tiny\red $2$}}
\psline[linestyle=solid,linewidth=0.05,linecolor=red](3.5,0)(3.5,-0.15)
\rput(3.425,0.4){ {\tiny\red $1$}}
\end{pspicture} 
\psset{unit=0.6cm}
\begin{pspicture}(-1.95,-0.45)(6.95,7.3782)
\pspolygon[linecolor=white,fillstyle=solid,fillcolor=backgroundgray,linearc=0.3](-1.95,-0.45)(6.95,-0.45)(6.95,7.3782)(-1.95,7.3782)

\psset{linewidth=0.5pt,linecolor=gray,linestyle=solid,fillstyle=none}
\pspolygon[fillstyle=solid,fillcolor=Apricot,linecolor=white](0,0)(4,0)(6.5,4.3301)(5,6.9282)(1,6.9282)(-1.5,2.5981)
\psline[linecolor=white](-1,1.7321)(2,6.9282)
\psline[linecolor=white](-0.5,0.86603)(3,6.9282)
\psline[linecolor=white](0,0)(4,6.9282)
\psline[linecolor=white](1,0)(5,6.9282)
\psline[linecolor=white](2,0)(5.5,6.0622)
\psline[linecolor=white](3,0)(6,5.1962)
\psline[linecolor=white](-1,3.4641)(1,0)
\psline[linecolor=white](-0.5,4.3301)(2,0)
\psline[linecolor=white](0,5.1962)(3,0)
\psline[linecolor=white](0.5,6.0622)(4,0)
\psline[linecolor=white](1,6.9282)(4.5,0.86603)
\psline[linecolor=white](2,6.9282)(5,1.7321)
\psline[linecolor=white](3,6.9282)(5.5,2.5981)
\psline[linecolor=white](4,6.9282)(6,3.4641)
\psset{fillstyle=solid,fillcolor=white,linecolor=white}
\pspolygon(2,0)(3,0)(2.5,0.86603)
\pspolygon(2.5,0.86603)(3,1.7321)(2,1.7321)
\pspolygon(2,1.7321)(3,1.7321)(2.5,2.5981)
\pspolygon(2.5,2.5981)(3,3.4641)(2,3.4641)
\pspolygon(2,3.4641)(3,3.4641)(2.5,4.3301)
\pspolygon(2.5,4.3301)(3,5.1962)(2,5.1962)
\pspolygon[fillstyle=solid,fillcolor=Tan,linecolor=white](3,0)(4,0)(4.5,0.86603)(3.5,0.86603)
\pspolygon[fillstyle=solid,fillcolor=Tan,linecolor=white](3.5,0.86603)(4.5,0.86603)(5,1.7321)(4,1.7321)
\pspolygon[fillstyle=solid,fillcolor=Mahogany,linecolor=white](4,1.7321)(5,1.7321)(4.5,2.5981)(3.5,2.5981)
\pspolygon[fillstyle=solid,fillcolor=Tan,linecolor=white](3.5,2.5981)(4.5,2.5981)(5,3.4641)(4,3.4641)
\pspolygon[fillstyle=solid,fillcolor=Tan,linecolor=white](4,3.4641)(5,3.4641)(5.5,4.3301)(4.5,4.3301)
\pspolygon[fillstyle=solid,fillcolor=Mahogany,linecolor=white](4.5,4.3301)(5.5,4.3301)(5,5.1962)(4,5.1962)
\pspolygon[fillstyle=solid,fillcolor=Tan,linecolor=white](4,5.1962)(5,5.1962)(5.5,6.0622)(4.5,6.0622)
\pspolygon[fillstyle=solid,fillcolor=Mahogany,linecolor=white](4.5,6.0622)(5.5,6.0622)(5,6.9282)(4,6.9282)
\pspolygon[fillstyle=solid,fillcolor=Mahogany,linecolor=white](1,0)(2,0)(1.5,0.86603)(0.5,0.86603)
\pspolygon[fillstyle=solid,fillcolor=Tan,linecolor=white](0.5,0.86603)(1.5,0.86603)(2,1.7321)(1,1.7321)
\pspolygon[fillstyle=solid,fillcolor=Mahogany,linecolor=white](1,1.7321)(2,1.7321)(1.5,2.5981)(0.5,2.5981)
\pspolygon[fillstyle=solid,fillcolor=Tan,linecolor=white](0.5,2.5981)(1.5,2.5981)(2,3.4641)(1,3.4641)
\pspolygon[fillstyle=solid,fillcolor=Tan,linecolor=white](1,3.4641)(2,3.4641)(2.5,4.3301)(1.5,4.3301)
\pspolygon[fillstyle=solid,fillcolor=Mahogany,linecolor=white](1.5,4.3301)(2.5,4.3301)(2,5.1962)(1,5.1962)
\pspolygon[fillstyle=solid,fillcolor=Tan,linecolor=white](1,5.1962)(2,5.1962)(2.5,6.0622)(1.5,6.0622)
\pspolygon[fillstyle=solid,fillcolor=Tan,linecolor=white](1.5,6.0622)(2.5,6.0622)(3,6.9282)(2,6.9282)
\pspolygon[fillstyle=solid,fillcolor=Mahogany,linecolor=white](0,0)(1,0)(0.5,0.86603)(-0.5,0.86603)
\pspolygon[fillstyle=solid,fillcolor=Mahogany,linecolor=white](-0.5,0.86603)(0.5,0.86603)(0,1.7321)(-1,1.7321)
\pspolygon[fillstyle=solid,fillcolor=Tan,linecolor=white](-1,1.7321)(0,1.7321)(0.5,2.5981)(-0.5,2.5981)
\pspolygon[fillstyle=solid,fillcolor=Tan,linecolor=white](-0.5,2.5981)(0.5,2.5981)(1,3.4641)(0,3.4641)
\pspolygon[fillstyle=solid,fillcolor=Mahogany,linecolor=white](0,3.4641)(1,3.4641)(0.5,4.3301)(-0.5,4.3301)
\pspolygon[fillstyle=solid,fillcolor=Tan,linecolor=white](-0.5,4.3301)(0.5,4.3301)(1,5.1962)(0,5.1962)
\pspolygon[fillstyle=solid,fillcolor=Tan,linecolor=white](0,5.1962)(1,5.1962)(1.5,6.0622)(0.5,6.0622)
\pspolygon[fillstyle=solid,fillcolor=Tan,linecolor=white](0.5,6.0622)(1.5,6.0622)(2,6.9282)(1,6.9282)
\pspolygon[fillstyle=solid,fillcolor=Tan,linecolor=white](2,5.1962)(3,5.1962)(3.5,6.0622)(2.5,6.0622)
\pspolygon[fillstyle=solid,fillcolor=Tan,linecolor=white](2.5,6.0622)(3.5,6.0622)(4,6.9282)(3,6.9282)
\psset{linewidth=1pt,linecolor=black,linestyle=solid,fillstyle=none}
\pspolygon[fillstyle=none](0,0)(4,0)(6.5,4.3301)(5,6.9282)(1,6.9282)(-1.5,2.5981)
\psset{linewidth=0.1,linecolor=blue,linearc=0.15}
\end{pspicture} 
\psset{unit=0.6cm}
\begin{pspicture}(-1.95,-0.575)(6.95,7.5032)
\pspolygon[linecolor=white,fillstyle=solid,fillcolor=backgroundgray,linearc=0.3](-1.95,-0.575)(6.95,-0.575)(6.95,7.5032)(-1.95,7.5032)

\psset{linewidth=0.5pt,linecolor=gray,linestyle=solid,fillstyle=none}
\pspolygon[fillstyle=solid,fillcolor=Apricot,linecolor=white](0,0)(4,0)(6.5,4.3301)(5,6.9282)(1,6.9282)(-1.5,2.5981)
\psline[linecolor=white](-1,1.7321)(2,6.9282)
\psline[linecolor=white](-0.5,0.86603)(3,6.9282)
\psline[linecolor=white](0,0)(4,6.9282)
\psline[linecolor=white](1,0)(5,6.9282)
\psline[linecolor=white](2,0)(5.5,6.0622)
\psline[linecolor=white](3,0)(6,5.1962)
\psline[linecolor=white](-1,3.4641)(1,0)
\psline[linecolor=white](-0.5,4.3301)(2,0)
\psline[linecolor=white](0,5.1962)(3,0)
\psline[linecolor=white](0.5,6.0622)(4,0)
\psline[linecolor=white](1,6.9282)(4.5,0.86603)
\psline[linecolor=white](2,6.9282)(5,1.7321)
\psline[linecolor=white](3,6.9282)(5.5,2.5981)
\psline[linecolor=white](4,6.9282)(6,3.4641)
\psset{fillstyle=solid,fillcolor=white,linecolor=white}
\pspolygon(2,0)(3,0)(2.5,0.86603)
\pspolygon(2.5,0.86603)(3,1.7321)(2,1.7321)
\pspolygon(2,1.7321)(3,1.7321)(2.5,2.5981)
\pspolygon(2.5,2.5981)(3,3.4641)(2,3.4641)
\pspolygon(2,3.4641)(3,3.4641)(2.5,4.3301)
\pspolygon(2.5,4.3301)(3,5.1962)(2,5.1962)
\pspolygon[fillstyle=solid,fillcolor=Tan,linecolor=white](3,0)(4,0)(4.5,0.86603)(3.5,0.86603)
\pspolygon[fillstyle=solid,fillcolor=Tan,linecolor=white](3.5,0.86603)(4.5,0.86603)(5,1.7321)(4,1.7321)
\pspolygon[fillstyle=solid,fillcolor=Mahogany,linecolor=white](4,1.7321)(5,1.7321)(4.5,2.5981)(3.5,2.5981)
\pspolygon[fillstyle=solid,fillcolor=Tan,linecolor=white](3.5,2.5981)(4.5,2.5981)(5,3.4641)(4,3.4641)
\pspolygon[fillstyle=solid,fillcolor=Tan,linecolor=white](4,3.4641)(5,3.4641)(5.5,4.3301)(4.5,4.3301)
\pspolygon[fillstyle=solid,fillcolor=Mahogany,linecolor=white](4.5,4.3301)(5.5,4.3301)(5,5.1962)(4,5.1962)
\pspolygon[fillstyle=solid,fillcolor=Tan,linecolor=white](4,5.1962)(5,5.1962)(5.5,6.0622)(4.5,6.0622)
\pspolygon[fillstyle=solid,fillcolor=Mahogany,linecolor=white](4.5,6.0622)(5.5,6.0622)(5,6.9282)(4,6.9282)
\pspolygon[fillstyle=solid,fillcolor=Mahogany,linecolor=white](1,0)(2,0)(1.5,0.86603)(0.5,0.86603)
\pspolygon[fillstyle=solid,fillcolor=Tan,linecolor=white](0.5,0.86603)(1.5,0.86603)(2,1.7321)(1,1.7321)
\pspolygon[fillstyle=solid,fillcolor=Mahogany,linecolor=white](1,1.7321)(2,1.7321)(1.5,2.5981)(0.5,2.5981)
\pspolygon[fillstyle=solid,fillcolor=Tan,linecolor=white](0.5,2.5981)(1.5,2.5981)(2,3.4641)(1,3.4641)
\pspolygon[fillstyle=solid,fillcolor=Tan,linecolor=white](1,3.4641)(2,3.4641)(2.5,4.3301)(1.5,4.3301)
\pspolygon[fillstyle=solid,fillcolor=Mahogany,linecolor=white](1.5,4.3301)(2.5,4.3301)(2,5.1962)(1,5.1962)
\pspolygon[fillstyle=solid,fillcolor=Tan,linecolor=white](1,5.1962)(2,5.1962)(2.5,6.0622)(1.5,6.0622)
\pspolygon[fillstyle=solid,fillcolor=Tan,linecolor=white](1.5,6.0622)(2.5,6.0622)(3,6.9282)(2,6.9282)
\pspolygon[fillstyle=solid,fillcolor=Mahogany,linecolor=white](0,0)(1,0)(0.5,0.86603)(-0.5,0.86603)
\pspolygon[fillstyle=solid,fillcolor=Mahogany,linecolor=white](-0.5,0.86603)(0.5,0.86603)(0,1.7321)(-1,1.7321)
\pspolygon[fillstyle=solid,fillcolor=Tan,linecolor=white](-1,1.7321)(0,1.7321)(0.5,2.5981)(-0.5,2.5981)
\pspolygon[fillstyle=solid,fillcolor=Tan,linecolor=white](-0.5,2.5981)(0.5,2.5981)(1,3.4641)(0,3.4641)
\pspolygon[fillstyle=solid,fillcolor=Mahogany,linecolor=white](0,3.4641)(1,3.4641)(0.5,4.3301)(-0.5,4.3301)
\pspolygon[fillstyle=solid,fillcolor=Tan,linecolor=white](-0.5,4.3301)(0.5,4.3301)(1,5.1962)(0,5.1962)
\pspolygon[fillstyle=solid,fillcolor=Tan,linecolor=white](0,5.1962)(1,5.1962)(1.5,6.0622)(0.5,6.0622)
\pspolygon[fillstyle=solid,fillcolor=Tan,linecolor=white](0.5,6.0622)(1.5,6.0622)(2,6.9282)(1,6.9282)
\pspolygon[fillstyle=solid,fillcolor=Tan,linecolor=white](2,5.1962)(3,5.1962)(3.5,6.0622)(2.5,6.0622)
\pspolygon[fillstyle=solid,fillcolor=Tan,linecolor=white](2.5,6.0622)(3.5,6.0622)(4,6.9282)(3,6.9282)
\psset{linewidth=1pt,linecolor=black,linestyle=solid,fillstyle=none}
\pspolygon[fillstyle=none](0,0)(4,0)(6.5,4.3301)(5,6.9282)(1,6.9282)(-1.5,2.5981)
\psset{linewidth=0.1,linecolor=blue,linearc=0.15}
\psline(3.5,0)(4.5,1.7321)(4,2.5981)(5,4.3301)(4.5,5.1962)(5,6.0622)(4.5,6.9282)
\psline(1.5,0)(1,0.86603)(1.5,1.7321)(1,2.5981)(2,4.3301)(1.5,5.1962)(2.5,6.9282)
\psline(0.5,0)(-0.5,1.7321)(0.5,3.4641)(0,4.3301)(1.5,6.9282)
\psline(2.5,5.1962)(3.5,6.9282)
\pscircle[fillstyle=solid,linecolor=black,linewidth=0.5pt,fillcolor=red,](3.5,0){0.125}
\pscircle[fillstyle=solid,linecolor=black,linewidth=0.5pt,fillcolor=red,](2.5,0){0.125}
\pscircle[fillstyle=solid,linecolor=black,linewidth=0.5pt,fillcolor=red,](1.5,0){0.125}
\pscircle[fillstyle=solid,linecolor=black,linewidth=0.5pt,fillcolor=red,](0.5,0){0.125}
\pscircle[fillstyle=solid,linecolor=black,linewidth=0.5pt,fillcolor=red,](2.5,1.7321){0.125}
\pscircle[fillstyle=solid,linecolor=black,linewidth=0.5pt,fillcolor=red,](2.5,3.4641){0.125}
\pscircle[fillstyle=solid,linecolor=black,linewidth=0.5pt,fillcolor=red,](2.5,5.1962){0.125}
\pscircle[fillstyle=solid,linecolor=black,linewidth=0.5pt,fillcolor=red,fillcolor=green](4.5,6.9282){0.125}
\pscircle[fillstyle=solid,linecolor=black,linewidth=0.5pt,fillcolor=red,fillcolor=green](3.5,6.9282){0.125}
\pscircle[fillstyle=solid,linecolor=black,linewidth=0.5pt,fillcolor=red,fillcolor=green](2.5,6.9282){0.125}
\pscircle[fillstyle=solid,linecolor=black,linewidth=0.5pt,fillcolor=red,fillcolor=green](1.5,6.9282){0.125}
\pswedge[fillstyle=solid,linecolor=black,linewidth=0.5pt,fillcolor=green](2.5,0){0.125}{0}{180}
\pswedge[fillstyle=solid,linecolor=black,linewidth=0.5pt,fillcolor=green](2.5,1.7321){0.125}{0}{180}
\pswedge[fillstyle=solid,linecolor=black,linewidth=0.5pt,fillcolor=green](2.5,3.4641){0.125}{0}{180}
\end{pspicture} 
\psset{unit=0.6cm}
\begin{pspicture}(-3.95,-0.95)(5.95,6.95)
\pspolygon[linecolor=white,fillstyle=solid,fillcolor=backgroundgray,linearc=0.3](-3.95,-0.95)(5.95,-0.95)(5.95,6.95)(-3.95,6.95)

\psset{linewidth=0.5pt,linecolor=gray,linestyle=solid,fillstyle=none}
\psline(-3,-0.5)(-3,6.5)
\psline(-2,-0.5)(-2,6.5)
\psline(-1,-0.5)(-1,6.5)
\psline(0,-0.5)(0,6.5)
\psline(1,-0.5)(1,6.5)
\psline(2,-0.5)(2,6.5)
\psline(3,-0.5)(3,6.5)
\psline(4,-0.5)(4,6.5)
\psline(5,-0.5)(5,6.5)
\psline(-3.5,1)(5.5,1)
\psline(-3.5,2)(5.5,2)
\psline(-3.5,3)(5.5,3)
\psline(-3.5,4)(5.5,4)
\psline(-3.5,5)(5.5,5)
\psline(-3.5,6)(5.5,6)
\psset{linewidth=1pt,linecolor=black,linestyle=solid,fillstyle=none}
\psline{->}(-3.5,0)(5.5,0)
\psline{->}(0,-0.5)(0,6.5)
\psset{linewidth=0.1,linecolor=blue,linearc=0.1}
\psline(0,0)(2,0)(2,1)(4,1)(4,2)(5,2)(5,3)
\psline(-2,2)(-2,3)(-1,3)(-1,4)(1,4)(1,5)(3,5)
\psline(-3,3)(-3,5)(-1,5)(-1,6)(2,6)
\psline(2,4)(4,4)
\pscircle[fillstyle=solid,linecolor=black,linewidth=0.5pt,fillcolor=red,](0,0){0.125}
\pscircle[fillstyle=solid,linecolor=black,linewidth=0.5pt,fillcolor=red,](-1,1){0.125}
\pscircle[fillstyle=solid,linecolor=black,linewidth=0.5pt,fillcolor=red,](-2,2){0.125}
\pscircle[fillstyle=solid,linecolor=black,linewidth=0.5pt,fillcolor=red,](-3,3){0.125}
\pscircle[fillstyle=solid,linecolor=black,linewidth=0.5pt,fillcolor=red,](0,2){0.125}
\pscircle[fillstyle=solid,linecolor=black,linewidth=0.5pt,fillcolor=red,](1,3){0.125}
\pscircle[fillstyle=solid,linecolor=black,linewidth=0.5pt,fillcolor=red,](2,4){0.125}
\pscircle[fillstyle=solid,linecolor=black,linewidth=0.5pt,fillcolor=red,fillcolor=green](5,3){0.125}
\pscircle[fillstyle=solid,linecolor=black,linewidth=0.5pt,fillcolor=red,fillcolor=green](4,4){0.125}
\pscircle[fillstyle=solid,linecolor=black,linewidth=0.5pt,fillcolor=red,fillcolor=green](3,5){0.125}
\pscircle[fillstyle=solid,linecolor=black,linewidth=0.5pt,fillcolor=red,fillcolor=green](2,6){0.125}
\pswedge[fillstyle=solid,linecolor=black,linewidth=0.5pt,fillcolor=green](-1,1){0.125}{-45}{135}
\pswedge[fillstyle=solid,linecolor=black,linewidth=0.5pt,fillcolor=green](0,2){0.125}{-45}{135}
\pswedge[fillstyle=solid,linecolor=black,linewidth=0.5pt,fillcolor=green](1,3){0.125}{-45}{135}
\end{pspicture} %
}{
The hexagon with side lengths $\pas{a,b,c}=\pas{4,5,3}$ and \EM{odd intrusion} of length $d=3$ at position $p=2$.%
}{
The upper left picture shows the hexagon with side lengths $\pas{a,b,c}=\pas{4,5,3}$ in the triangular
lattice with an \EM{odd intrusion} of length $d=3$ (marked as gray triangles) at position $2$ (possible positions of
intrusions are indicated by ticks at the base line of the hexagon). The upper right
picture shows a \EM{lozenge tiling} of this hexagonal region, and the lower left picture
shows the \EM{same} tiling together with the corresponding family of \EM{nonintersecting lattice paths}
(starting points of the paths are coloured red, and ending points are coloured green; points which
are starting \EM{and} ending points --- these correspond to the ``intrusion'' --- are coloured red \EM{and} green):
It is a well--known fact that this correspondence is a \EM{bijection}.

The lower right picture shows the family of nonintersecting lattice paths in the
integer lattice $\Z\times\Z$: These are obtained by tilting the paths shown in the
picture to the left and shifting them in the plane such that the lowest starting
point coincides with the origin $\pas{0,0}$. Altogether, this gives $a=4$ \EM{lateral}
starting points plus $d=3$ \EM{intrusive} starting points
$$
\pas{
	\underbrace{(0,0),(-1,1),(-2,2),(-3,3)}_{\text{lateral}},
	\underbrace{(0,2),(1,3),(2,4)}_{\text{intrusive}}
},
$$
and $a=4$ \EM{lateral} ending points plus $d=3$ \EM{intrusive} ending points
$$\pas{
	\underbrace{(5,3),(4,4),(3,5),(2,6)}_{\text{lateral}},
	\underbrace{(-1,1),(0,2),(1,3)}_{\text{intrusive}}
}.
$$

}{
hexd%
}

In the \EM{triangular lattice},
we consider \EM{$\pas{a,b,c}$--hexagons} with side lengths $a,b,c$, $a,b,c$ (anti--clockwise,
with $a,b,c\in\N$), see the upper left pictures in Figures \ref{fig:hexc} and \ref{fig:hexd}.

We assume that the triangular lattice is drawn in a way that the hexagon's \EM{baseline} of length
$a$ appears \EM{horizontal}, and that an even number of \EM{vertically}
stacked triangles, adjacent to the baseline, is removed from the hexagon.
Following Byun's wording, we call
such stack of $2d$ removed triangles an \EM{intrusion} of length $d$. Intrusions
come in two flavours, namely
\bit
\item starting with a triangle which has only a single vertex in common with
	the hexagon's baseline, see the upper left picture in \figref{hexc}: We shall call this
	type an \EM{even} intrusion;
\item or starting with a triangle having an edge in common with the hexagon's baseline,
	see the upper left picture in \figref{hexd}: We shall call this type
	an \EM{odd} intrusion.
\eit
We count the horizontal \EM{position} $p$ of intrusions from right to left, starting with
$0$ for even intrusions and starting with $1$ for odd intrusions, see the upper left
pictures in Figures \ref{fig:hexc} and \ref{fig:hexd}. We shall call such hexagon with
an intrusion a \EM{damaged hexagon}. 

\secB{Lozenge tilings and their enumeration}
A \EM{lozenge} is a geometric shape in the triangular lattice
which covers two triangles sharing a common edge.

A \EM{lozenge tiling} of some (damaged) hexagon is a set
of \EM{pairwise disjoint} lozenges (in the sense that no two lozenges have a triangle
in common) which together cover \EM{all} triangles of the damaged hexagon, see
the upper right pictures in Figures \ref{fig:hexc} and \ref{fig:hexd}.

We denote by $\evencount{a,b,c,d,p}$ or
$\oddcount{a,b,c,d,p}$, respectively, the number of lozenge tilings of the damaged
$\pas{a,b,c}$--hexagon
with an even or odd, respectively, intrusion of length $d$ in position $p$.

The \EM{enumeration} of lozenge tilings of hexagonal regions in the triangular
lattice often leads to interesting formulas, the most prominent of which is
MacMahon's formula \cite[\S~429]{MacMahon:1916:CA2} giving the number of all lozenge tilings of the
$\pas{a,b,c}$--hexagon (without damage, i.e., with an intrusion of length $0$).
Denoting this number by $\macmahon{a,b,c}$, we have
\begin{align}
\macmahon{a,b,c}
&= 
\evencount{a,b,c,0,p} = \oddcount{a,b,c,0,p} \notag\\  
&=
\det\brk{\binom{b+c}{b-i+j}}_{i,j=1}^a \notag\\
&=
\prod_{i=1}^a\prod_{j=1}^b\prod_{k=1}^c\frac{i + j + k - 1}{i + j + k - 2}\notag\\
&=
\prod_{i=0}^{a-1}\frac{i!\pas{b+c+i}!}{\pas{b+i}!\pas{c+i}!}.\label{eq:macmahon}
\end{align}
(The expression of this number as a determinant will become clear in section \ref{sec:determinants},
and section~\ref{eq:mm-proof} contains a short proof of MacMahon's formula.)

From now on, letters $a,b,c\in\N$ will always denote the side lengths of some
hexagon, and letters $d\in\N$ and $p\in\Z$ will always denote the length and position of an intrusion.

\secB{Byun's formulas}
Byun  found and proved nice formulas for the special cases
\bit
\item $a=2p$ for even intrusions \cite[equation (2.1)]{Byun:2022:LTOAHWAHI},
\item and $a=2p+1$ for odd intrusions \cite[equation (2.2)]{Byun:2022:LTOAHWAHI}.
\eit
In the notation just introduced, \cite[equation (2.1)]{Byun:2022:LTOAHWAHI} is equivalent to
\begin{equation}
\label{eq:byun-even}
\evencount{2p,b,c,d,p} =
\macmahon{2p,b,c}\cdot
	\prod_{k=1}^{d}
	4^p\frac{
		\pochhammer{1 + b - k}{p}\pochhammer{1 + c - k}{p}\pochhammer{-\frac12 + k}{p}
	}{
		\pochhammer{2 + b + c - 2k}{2p}\pochhammer{k}{p}
	},
\end{equation}
and \cite[equation(2.2)]{Byun:2022:LTOAHWAHI} reads
{\tiny
\begin{multline}
\label{eq:byun-odd}
\oddcount{2p+1,b,c,d,p} =
\macmahon{2p+1,b,c}\cdot\frac1{4^d}\cdot \\
	\prod_{k=0}^{d-1}
	\frac{
		\pochhammer{a+k+1}{c-2k}
		\pochhammer{k+\frac32}{c-2k-2}
		\pochhammer{b-k}{\floor{\frac{c-b}2}}
		\pochhammer{c-k-\frac12}{-\floor{\frac{c-b}2}}
	}{
		\pochhammer{k+1}{c-2k-1}
		\pochhammer{a+k+\frac32}{c-2k-1}
		\pochhammer{a+b-k+1}{\floor{\frac{c-b}2}}
		\pochhammer{a+c-k+\frac12}{-\floor{\frac{c-b}2}}
	}.
\end{multline}
}

Byun's proofs of these formula involved results by Ciucu
(\cite[Theorem 3.1]{Ciucu:2005:AGOMF} and \cite[Matching Factorisation
Theorem]{Ciucu:1997:EOPMIGWRS})
and certain elegant recursions for the enumeration of perfect matchings
(basically applications of Pfaffian identities to the Kasteleyn--Percus
method \cite{kasteleyn:1963,percus:1969}, for which Kuo \cite{kuo:2004}
coined the name ``graphical condensation'').

\secA{Enumeration of lozenge tilings by determinants}
\label{sec:determinants}
\secB{Bijection between lozenge tilings and \nilp s} Lozenge tilings
are in bijection with \EM{\nilp s}. Instead of giving a formal description
we point to the lower left picture in \figref{hexc}: 
First, observe that a lozenge tilings of an $\pas{a,b,c}$--hexagon
with an
\EM{even} intrusion might be viewed as a ``stack of cubes'' fitting in a
rectangular box with side lengths ${a,b,c}$, and that such ``stack of cubes''
is uniquely described by a family of lattice paths, where the intrusion
of length $d$ corresponds to $d$ lattice paths of length $0$. These lattice
paths and their respective starting and ending points are
indicated in the lower left picture of \figref{hexc} by blue lines and
by red and green points, and it is easy
to see that by tilting the picture, the paths appear in the integer lattice
$\Z\times \Z$, with unit steps directed upwards and to the right (see the lower
right picture in \figref{hexc}).

Note that there are starting and ending points
\bit
\item on the horizontal sides of the hexagon: We shall call these \EM{lateral} points,
\item and inside the intrusion's removed triangles: We shall call these \EM{intrusive}
points.
\eit

The situation is a little bit more complicated in the case of odd intrusion,
since lozenge tilings do not correspond to a simple ``stack of cubes'' now:
But it is easy to see that there is basically the same bijection with
\nilp s, see \figref{hexd}.

\secB{Counting \nilp s with determinants}
Of course, we may shift the \nilp s in the integer lattice $\Z\times\Z$ such that
the \EM{lowest lateral starting point} has coordinates $\pas{0,0}$: Then the
coordinates of the \EM{lateral} starting and ending points, counted from
right to left, are the following:
\bit
\item For the $i$--th lateral starting: $\pas{1-i,i-1}$,
\item and for the $j$--th lateral ending point: $\pas{b+1-j,c+j-1}$.
\eit
The coordinates of the \EM{intrusive} starting and ending points are the following:
\bit
\item for even intrusions, the $i$--th intrusive starting point
 	coincides with the $i$--th intrusive ending point:
	$\pas{-p+i,+i-1}$,
\item for $p$ intrusions, 
	\bit
	\item the $i$--th intrusive starting point: $\pas{-p+i,p+i}$,
	\item and the $j$--th intrusive ending point: $\pas{-p+j-1,p+j-1}$.
	\eit
\eit
(See again Figures~\ref{fig:hexc} and \ref{fig:hexd}.)

The well--known Lindstr\"om--Gessel--Viennot method \cite{Lindstroem:1973:OTVROIM,Gessel-Viennot:1998:DPAPP}
counts the number of \nilp s in the integer lattice $\Z\times\Z$ as a \EM{determinant},
whose $\pas{i,j}$--entry equals the number of lattice paths from the $i$--th starting
point to the $j$--th starting point (under the assumption that all permutations $\pi$ for which
there actually \EM{are} \nilp s from starting point $i$ to ending point $\pi\of i$ have the same
\EM{positive} sign).

We shall consider the following order of starting and ending points of the
\nilp s corresponding to lozenge tilings with even or odd intrusions:
\bit
\item First, there come the \EM{lateral} points, numbered from lower right to upper left,
\item then, there come the \EM{intrusive} points, numbered from lower left to upper right.
\eit
Note that there is precisely one permutation $\pi$ admitting \nilp s running from
starting point $i$ to ending point $\pi\of i$: For even intrusions, this is simply the
identity permutation, while for odd intrusions the corresponding permutation might have
the negative sign. 

Clearly, the number of all lattice paths in the integer lattice $\Z\times\Z$
starting in $\pas{x,y}$ and ending
in $\pas{u,v}$, with unit steps to the right and upwards, is either zero or a binomial coefficient. By slight abuse of the
standard notation, throughout this paper  we adopt the convention
$$
\binom nk \equiv 0 \text{ if } k < 0 \text{ or } k > n
$$
(i.e., $\binom{-1}{3}=0$, not $-1$)
and set $n=\pas{u-x}+\pas{v-y}$ and $k=\pas{u-x}$: Then this number of lattice paths is simply
$\binom nk = \binom{u-x+v-y}{u-x}$.

\begin{ex}
\label{ex:hexc_gf}
For the damaged hexagon with parameters $\pas{a,b,c,d,p}=\pas{4,5,3,2,4}$
depicted in \figref{hexc}, the determinant counting the \nilp s (and thus
the lozenge tilings) is
{\tiny
$$
\det\left[\begin{array}{cccc|cc}
\binom{8}{5} & \binom{8}{4} & \binom{8}{3} & \binom{8}{2} & {\gray 0} & {\gray \binom{3}{0}} \\[1mm]
\binom{8}{6} & \binom{8}{5} & \binom{8}{4} & \binom{8}{3} & {\gray \binom{1}{0}} & {\gray \binom{3}{1}} \\[1mm]
\binom{8}{7} & \binom{8}{6} & \binom{8}{5} & \binom{8}{4} & {\gray \binom{1}{1}} & {\gray \binom{3}{2}} \\[1mm]
\binom{8}{8} & \binom{8}{7} & \binom{8}{6} & \binom{8}{5} & {\gray 0} & {\gray \binom{3}{3}} \\[1mm]
\hline
{\gray \binom{7}{6}} & {\gray \binom{7}{5}} & {\gray \binom{7}{4}} & {\gray \binom{7}{3}} & {\gray \binom{0}{0}} & {\gray \binom{2}{1}} \\[1mm]
{\gray \binom{5}{5}} & {\gray \binom{5}{4}} & {\gray \binom{5}{3}} & {\gray \binom{5}{2}} & {\gray 0} & {\gray \binom{0}{0}} \\[1mm]
\end{array}\right]
=
\det\left[\begin{array}{cccc|cc}
56 & 70 & 56 & 28 & {\gray 0} & {\gray 1} \\
28 & 56 & 70 & 56 & {\gray 1} & {\gray 3} \\
8 & 28 & 56 & 70 & {\gray 1} & {\gray 3} \\
1 & 8 & 28 & 56 & {\gray 0} & {\gray 1} \\
\hline
{\gray 7} & {\gray 21} & {\gray 35} & {\gray 35} & {\gray 1} & {\gray 2} \\
{\gray 1} & {\gray 5} & {\gray 10} & {\gray 10} & {\gray 0} & {\gray 1} \\
\end{array}\right] = 12600.
$$

}
\end{ex}

\begin{ex}
For the damaged hexagon with parameters $\pas{a,b,c,d,p}=\pas{4,5,3,3,3}$
depicted in \figref{hexd}, the determinant counting the \nilp s (and thus
the lozenge tilings) is
{\tiny
$$
\det\left[\begin{array}{cccc|ccc}
\binom{8}{5} & \binom{8}{4} & \binom{8}{3} & \binom{8}{2} & {\gray 0} & {\gray \binom{2}{0}} & {\gray \binom{4}{1}} \\[1mm]
\binom{8}{6} & \binom{8}{5} & \binom{8}{4} & \binom{8}{3} & {\gray \binom{0}{0}} & {\gray \binom{2}{1}} & {\gray \binom{4}{2}} \\[1mm]
\binom{8}{7} & \binom{8}{6} & \binom{8}{5} & \binom{8}{4} & {\gray 0} & {\gray \binom{2}{2}} & {\gray \binom{4}{3}} \\[1mm]
\binom{8}{8} & \binom{8}{7} & \binom{8}{6} & \binom{8}{5} & {\gray 0} & {\gray 0} & {\gray \binom{4}{4}} \\[1mm]
\hline
{\gray \binom{6}{5}} & {\gray \binom{6}{4}} & {\gray \binom{6}{3}} & {\gray \binom{6}{2}} & {\gray 0} & {\gray \binom{0}{0}} & {\gray \binom{2}{1}} \\[1mm]
{\gray \binom{4}{4}} & {\gray \binom{4}{3}} & {\gray \binom{4}{2}} & {\gray \binom{4}{1}} & {\gray 0} & {\gray 0} & {\gray \binom{0}{0}} \\[1mm]
{\gray 0} & {\gray \binom{2}{2}} & {\gray \binom{2}{1}} & {\gray \binom{2}{0}} & {\gray 0} & {\gray 0} & {\gray 0} \\[1mm]
\end{array}\right]
=
\det\left[\begin{array}{cccc|ccc}
56 & 70 & 56 & 28 & {\gray 0} & {\gray 1} & {\gray 4} \\
28 & 56 & 70 & 56 & {\gray 1} & {\gray 2} & {\gray 6} \\
8 & 28 & 56 & 70 & {\gray 0} & {\gray 1} & {\gray 4} \\
1 & 8 & 28 & 56 & {\gray 0} & {\gray 0} & {\gray 1} \\
\hline
{\gray 6} & {\gray 15} & {\gray 20} & {\gray 15} & {\gray 0} & {\gray 1} & {\gray 2} \\
{\gray 1} & {\gray 4} & {\gray 6} & {\gray 4} & {\gray 0} & {\gray 0} & {\gray 1} \\
{\gray 0} & {\gray 1} & {\gray 2} & {\gray 1} & {\gray 0} & {\gray 0} & {\gray 0} \\
\end{array}\right] = -4032.
$$

}

Note that this determinant is \EM{negative}: This does, of course, not mean that the
number of lozenge tilings is negative, but that the permutation admitting \nilp s
has the negative sign.
\end{ex}

\secB{Symmetries and Dodgson's condensation formula}
Note that our definition of starting and ending points makes perfect sense also
for intrusions in positions $p<0$ or $p>a$ (but for \EM{odd} intrusions,
positions outside
the range $\brk{1,a}$ would give an endpoint which cannot be reached by
\EM{any} of the starting points, hence the number of \nilp s and the corresponding
determinant is zero). From now on, we shall understand
$\evencount{a,b,c,d,p}$ and $\oddcount{a,b,c,d,p}$ as notations for the
\EM{determinants} described above (i.e., $p<0$ or $p>a$ is now possible, and
$\oddcount{a,b,c,d,p}$ might give the \EM{negative} of the number of corresponding
lozenge tilings). 

By reflecting the damaged hexagon at a vertical axis, we observe the following symmetries:
\begin{equation}
\evencount{a,b,c,d,p} = \evencount{a,c,b,d,a-p}
\text{ and } 
\oddcount{a,b,c,d,p} = \oddcount{a,c,b,d,a-p+1}. \label{eq:symmetry}
\end{equation}

Now recall \EM{Dodgson's condensation formula} \cite{dodgson:1866} (also known as
Des\-nanot–-Jacobi’s Adjoint Matrix Theorem: According to \cite{bressoud:1999}, Lagrange
discovered this Theorem for dimension $n=3$, Desnanot proved it for dimensions $n\leq6$,
and Jacobi published the general theorem \cite{jacobi:1841a}, see also
\cite[vol. I, pp. 142]{muir:1906}): Let $M$ be some $n\times n$ matrix. Consider row 
indices $1\leq i_1\neq i_2 \leq n$ and column indices $1\leq j_1\neq j_2 \leq n$, and denote
by $M_{\pas{r}\vert\pas{c}}$ the matrix obtained from $M$ by deleting rows and columns with indices in lists
$\pas{r}$ and $\pas{c}$, respectively. Then there holds:
\begin{multline}
\label{eq:jacobi}
\det M \cdot \det M_{\pas{i_1,i_2}\vert\pas{j_1,j_2}} = \\
\det M_{\pas{i_1}\vert\pas{j_1}} \cdot \det M_{\pas{i_2}\vert\pas{j_2}} - 
\det M_{\pas{i_1}\vert\pas{j_2}} \cdot \det M_{\pas{i_2}\vert\pas{j_1}}.
\end{multline}

Applying Dodgson's condensation formula \eqref{eq:jacobi} to the determinant
$\evencount{a,b,c,d,p}$ 
for row and column indices
$i_1 = j_1 = 1$ and $i_2 = j_2 = a$ gives the following functional equation
{\small
\begin{align}
\evencount{a,b,c,d,p}\cdot\evencount{a-2,b,c,d,p-1}
&=
\evencount{a-1,b,c,d,p-1}\cdot\evencount{a-1,b,c,d,p} \notag \\
&-
\evencount{a-1,b+1,c-1,d,p-1}\cdot\evencount{a-1,b-1,c+1,d,p}\label{eq:dodgson}
\end{align}
}%
for all $d\in\N$, $p\in\Z$ and $a\geq 2$, and the analogous identity for $\oddcount{a,b,c,d,p}$
{\small
\begin{align}
\oddcount{a,b,c,d,p}\cdot\oddcount{a-2,b,c,d,p-1}
&=
\oddcount{a-1,b,c,d,p-1}\cdot\oddcount{a-1,b,c,d,p} \notag \\
&-
\oddcount{a-1,b+1,c-1,d,p-1}\cdot\oddcount{a-1,b-1,c+1,d,p}.\label{eq:dodgson-odd}
\end{align}
}%

By convention, empty determinants or products are equal to $1$, which is perfectly in line
with the fact that a (damaged) hexagon with side $a=0$ has, in fact, 
precisely one lozenge tiling:
$$
\evencount{0,b,c,d,p} \equiv 1.
$$

\section{Even intrusions}
\label{sec:even-intrusions}
In the rest of this paper, we shall restrict our considerations to even intrusions.

It is clear that even intrusions at positions ``too far away'' from the hexagon's baseline
are equivalent to ``no intrusions at all'' (as far as the counting of
lozenge tilings or \nilp s is concerned). More precisely: 
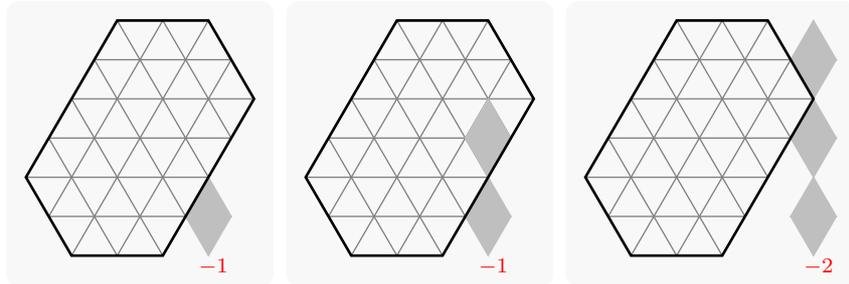
\begin{multline}
\label{eq:p-too-far}
\evencount{a,b,c,d,p} = \evencount{a,b,c,0,p} = \macmahon{a,b,c} \\
\text{{\small if $p\leq\max\of{-d,-\floor{\frac{b+1}2}}$ or $p\geq\min\of{a+d,a+\floor{\frac{c+1}2}}$}}.
\end{multline}
(See \figref{d1pminus} for an illustration.)

\psfigurescommented{
\psset{unit=0.6cm}
\begin{pspicture}(-1.45,-0.7)(4.45,5.6462)
\pspolygon[linecolor=white,fillstyle=solid,fillcolor=backgroundgray,linearc=0.3](-1.45,-0.7)(4.45,-0.7)(4.45,5.6462)(-1.45,5.6462)

\psset{linewidth=0.5pt,linecolor=gray,linestyle=solid,fillstyle=none}
\psline(-1,1.7321)(1,5.1962)
\psline(-0.5,0.86603)(2,5.1962)
\psline(0,0)(3,5.1962)
\psline(1,0)(3.5,4.3301)
\psline(2,0)(4,3.4641)
\psline(0,0)(-1,1.7321)
\psline(1,0)(-0.5,2.5981)
\psline(2,0)(0,3.4641)
\psline(2.5,0.86603)(0.5,4.3301)
\psline(3,1.7321)(1,5.1962)
\psline(3.5,2.5981)(2,5.1962)
\psline(4,3.4641)(3,5.1962)
\psline(2,0)(0,0)
\psline(2.5,0.86603)(-0.5,0.86603)
\psline(3,1.7321)(-1,1.7321)
\psline(3.5,2.5981)(-0.5,2.5981)
\psline(4,3.4641)(0,3.4641)
\psline(3.5,4.3301)(0.5,4.3301)
\psline(3,5.1962)(1,5.1962)
\psset{fillstyle=solid,fillcolor=lightgray,linecolor=lightgray}
\pspolygon(3,0)(3.5,0.86603)(2.5,0.86603)
\pspolygon(2.5,0.86603)(3.5,0.86603)(3,1.7321)
\rput(3,-0.25){ {\red $\scriptstyle -1$}}
\psset{linewidth=1pt,linecolor=black,linestyle=solid,fillstyle=none}
\pspolygon[fillstyle=none](0,0)(2,0)(4,3.4641)(3,5.1962)(1,5.1962)(-1,1.7321)
\end{pspicture} 
\psset{unit=0.6cm}
\begin{pspicture}(-1.45,-0.7)(4.45,5.6462)
\pspolygon[linecolor=white,fillstyle=solid,fillcolor=backgroundgray,linearc=0.3](-1.45,-0.7)(4.45,-0.7)(4.45,5.6462)(-1.45,5.6462)

\psset{linewidth=0.5pt,linecolor=gray,linestyle=solid,fillstyle=none}
\psline(-1,1.7321)(1,5.1962)
\psline(-0.5,0.86603)(2,5.1962)
\psline(0,0)(3,5.1962)
\psline(1,0)(3.5,4.3301)
\psline(2,0)(4,3.4641)
\psline(0,0)(-1,1.7321)
\psline(1,0)(-0.5,2.5981)
\psline(2,0)(0,3.4641)
\psline(2.5,0.86603)(0.5,4.3301)
\psline(3,1.7321)(1,5.1962)
\psline(3.5,2.5981)(2,5.1962)
\psline(4,3.4641)(3,5.1962)
\psline(2,0)(0,0)
\psline(2.5,0.86603)(-0.5,0.86603)
\psline(3,1.7321)(-1,1.7321)
\psline(3.5,2.5981)(-0.5,2.5981)
\psline(4,3.4641)(0,3.4641)
\psline(3.5,4.3301)(0.5,4.3301)
\psline(3,5.1962)(1,5.1962)
\psset{fillstyle=solid,fillcolor=lightgray,linecolor=lightgray}
\pspolygon(3,0)(3.5,0.86603)(2.5,0.86603)
\pspolygon(2.5,0.86603)(3.5,0.86603)(3,1.7321)
\pspolygon(3,1.7321)(3.5,2.5981)(2.5,2.5981)
\pspolygon(2.5,2.5981)(3.5,2.5981)(3,3.4641)
\rput(3,-0.25){ {\red $\scriptstyle -1$}}
\psset{linewidth=1pt,linecolor=black,linestyle=solid,fillstyle=none}
\pspolygon[fillstyle=none](0,0)(2,0)(4,3.4641)(3,5.1962)(1,5.1962)(-1,1.7321)
\end{pspicture} 
\psset{unit=0.6cm}
\begin{pspicture}(-1.45,-0.7)(4.95,5.6462)
\pspolygon[linecolor=white,fillstyle=solid,fillcolor=backgroundgray,linearc=0.3](-1.45,-0.7)(4.95,-0.7)(4.95,5.6462)(-1.45,5.6462)

\psset{linewidth=0.5pt,linecolor=gray,linestyle=solid,fillstyle=none}
\psline(-1,1.7321)(1,5.1962)
\psline(-0.5,0.86603)(2,5.1962)
\psline(0,0)(3,5.1962)
\psline(1,0)(3.5,4.3301)
\psline(2,0)(4,3.4641)
\psline(0,0)(-1,1.7321)
\psline(1,0)(-0.5,2.5981)
\psline(2,0)(0,3.4641)
\psline(2.5,0.86603)(0.5,4.3301)
\psline(3,1.7321)(1,5.1962)
\psline(3.5,2.5981)(2,5.1962)
\psline(4,3.4641)(3,5.1962)
\psline(2,0)(0,0)
\psline(2.5,0.86603)(-0.5,0.86603)
\psline(3,1.7321)(-1,1.7321)
\psline(3.5,2.5981)(-0.5,2.5981)
\psline(4,3.4641)(0,3.4641)
\psline(3.5,4.3301)(0.5,4.3301)
\psline(3,5.1962)(1,5.1962)
\psset{fillstyle=solid,fillcolor=lightgray,linecolor=lightgray}
\pspolygon(4,0)(4.5,0.86603)(3.5,0.86603)
\pspolygon(3.5,0.86603)(4.5,0.86603)(4,1.7321)
\pspolygon(4,1.7321)(4.5,2.5981)(3.5,2.5981)
\pspolygon(3.5,2.5981)(4.5,2.5981)(4,3.4641)
\pspolygon(4,3.4641)(4.5,4.3301)(3.5,4.3301)
\pspolygon(3.5,4.3301)(4.5,4.3301)(4,5.1962)
\rput(4,-0.25){ {\red $\scriptstyle -2$}}
\psset{linewidth=1pt,linecolor=black,linestyle=solid,fillstyle=none}
\pspolygon[fillstyle=none](0,0)(2,0)(4,3.4641)(3,5.1962)(1,5.1962)(-1,1.7321)
\end{pspicture} 
}{
Hexagons with side lengths $\pas{a,b,c}=\pas{2,4,2}$ and intrusions  of length $d\leq3$ at positions $p<0$.%
}{
All pictures show hexagons with side lengths $\pas{a,b,c}=\pas{2,4,2}$ and intrusions,
not all of which actually cause a \EM{damage} to the hexagon.
The left picture shows the situation $d=1$ and position $p=-1$, which illustrates the
fact that intrusions with $d\leq -p$ do not affect the number of tilings at all, and the same holds
for $-p\geq \frac b2$, as is illustrated in the right picture (where $d=3$ and $p=-2$). The central picture
illustrates the case $-p<\min\of{d,\frac b2}$ 
(i.e., 
where the intrusion ``actually causes damage''. 
The special case $p=1-d$ of this situation is considered in Proposition~\ref{pro:p=1-d}.
}{
d1pminus%
}

\secB{Simple observations: Cancellations}
\label{sec:cancellations}

Loosely speaking, Mac\-Mahon's formula \eqref{eq:macmahon} for the $\pas{a,b,c}$--hexagon
is a product of quotients of factorials. Hence it is clear that quotients of instances
of this formula will involve a lot of cancellations. For instance, by straightforward computation
we obtain:
\begin{align}
\frac{\macmahon{a,b,c}}{\macmahon{a-1,b,c}}
&=
\frac{\pas{a-1}!\pas{a+b+c-1}!}{\pas{a+b-1}!\pas{a+c-1}!}, \label{eq:cancel1} \\
\frac{\macmahon{a,b-1,c+1}}{\macmahon{a,b,c}}
&=
\frac{c!\pas{a+b-1}!}{\pas{a+c}!\pas{b-1}!}, \label{eq:cancel2} \\
\frac{\macmahon{a,b-1,c+1}}{\macmahon{a-1,b,c}}
&=
\frac{\pas{a-1}!c!\pas{a+b+c-1}!}{\pas{a+c}!\pas{a+c-1}!\pas{b-1}!}. \label{eq:cancel3}
\end{align}
Note that \eqref{eq:cancel2} and \eqref{eq:cancel3} give symmetric equations due to the obvious
symmetry
$$
\macmahon{a,b,c} = \macmahon{a,c,b}.
$$

\secB{A very general \EM{ansatz}}
\label{sec:general-ansatz}
We make the \EM{ansatz}
\begin{equation}
\label{eq:general_ansatz}
\evencount{a,b,c,d,p} = \macmahon{a,b,c}\cdot\generalfactor{a,b,c,d,p}.
\end{equation}
Substituting ansatz \eqref{eq:general_ansatz} in the recursion \eqref{eq:dodgson}
(derived from Dodgson's condensation formula \eqref{eq:jacobi}),
we obtain the following functional equation by straightforward cancellations (see the examples
of such cancellations in section~\ref{sec:cancellations}):
\begin{multline}
\label{eq:dodgson_general_ansatz}
(a-1)\cdot (a+b+c-1)\cdot \generalfactor{a-2,b,c,d,p-1}\cdot \generalfactor{a,b,c,d,p}= \\
(a+b-1)\cdot
   (a+c-1)\cdot \generalfactor{a-1,b,c,d,p-1}\cdot \generalfactor{a-1,b,c,d,p}-\\
   b\cdot c\cdot
   \generalfactor{a-1,b-1,c+1,d,p}\cdot \generalfactor{a-1,b+1,c-1,d,p-1}.
\end{multline}
For fixed $d$ and $p$, this amounts to the following recursive description
of $\generalfactor{a,b,c,d,p}$
\begin{multline}
\generalfactor{a,b,c,d,p}=\frac1{(a-1)\cdot (a+b+c-1)\cdot \generalfactor{a-2,b,c,d,p-1}}\\
\times\Bigl(
   (a+b-1)\cdot (a+c-1)\cdot \generalfactor{a-1,b,c,d,p-1}\cdot
   \generalfactor{a-1,b,c,d,p}-\\
   b\cdot c\cdot \generalfactor{a-1,b-1,c+1,d,p}\cdot
   \generalfactor{a-1,b+1,c-1,d,p-1}
\Bigr).\label{eq:general-recursion}
\end{multline}
Together with the boundary values
\bit
\item $\generalfactor{a,b,c,d,p-1}$ (for all 
$a, b,c$)
\item $\generalfactor{0,b,c,d,p}$ and $\generalfactor{1,b,c,d,p}$ (for all $b,c$),
\eit
the recursion \eqref{eq:general-recursion} \EM{uniquely} determines $\generalfactor{a,b,c,d,p}$
for all $a,b,c$.

\secC{First application: MacMahon's formula}
\label{eq:mm-proof}
These simple observations provide a short proof of MacMahon's formula \eqref{eq:macmahon}:
\begin{proof}[MacMahon's formula]
For $d=0$, we clearly have $\generalfactor{a,b,c,0,p} = \generalfactor{a,b,c,0,0}$ (i.e., the
position $p$ of an intrusion of length $0$ is irrelevant), and
\bit
\item $\generalfactor{0,b,c,0,0} = 1$ since $\evencount{0,b,c,0,0} = 1 = \macmahon{0,b,c}$
\item and $\generalfactor{1,b,c,0,0} = 1$ since $\evencount{1,b,c,0,0} = \binom{b+c}c = \macmahon{1,b,c}$.
\eit
MacMahon's formula is equivalent to $\generalfactor{a,b,c,0,0} \equiv 1$, and 
if we want to prove this equation, we simply have to show
that constant $1$ is a solution of the functional equation \eqref{eq:general_ansatz}.
But this amounts to the simple identity
$$
(a-1)\cdot (a+b+c-1)= 
(a+b-1)\cdot(a+c-1)-b\cdot c,
$$
which is immediately verified. 
\end{proof}

\secB{A very general plan of action}
So our quest for a formula giving $\evencount{a,b,c,d,p}$ leads us to the following
plan of action:
\begin{enumerate}
\item Let $d>0$ be fixed.
\item Identify some $p_0$ for which $\generalfactor{a,b,c,d,p_0}$
	can be derived easily. 
\item Observe that $\generalfactor{0,b,c,d,p}\equiv 1$, and find a formula giving
$\generalfactor{1,b,c,d,p}$ for all $p\geq p_0$.
\item Guess the formula giving $\generalfactor{a,b,c,d,p}$ and prove it by verifying that
it satisfies the functional equation
\eqref{eq:general-recursion}.
\end{enumerate}
From the geometric situation one might suspect that formulae
\bit
\item for $p\leq 0$
\item and for $0\leq p\leq a$
\eit
are of different quality (by the symmetry
$\evencount{a,b,c,d,p} = \evencount{a,c,b,d,a-p}$, we may omit the case $p\geq a$): Indeed, we
shall use a \EM{specialized} ansatz for the first case, and a \EM{modified} ansatz for the second case.

\secB{A specialized ansatz for $p\leq 0$}
For fixed $d>0$ and $p\leq 0$, we rewrite the function $\generalfactor{a,b,c,d,p}$ from \eqref{eq:general_ansatz} as follows:
\begin{equation}
\label{eq:special-ansatz}
\generalfactor{a,b,c,d,p} = \pas{\prod_{k=0}^{d+p-1}
\frac{\pochhammer{c-k}{a-d-p+1+2k}}{\pochhammer{b+c-2d+2k+2}{a+2d-2-3k}}}\cdot\specialfactor{a,b,c,d,p}
\end{equation}
Numerical experiments indicate that this specialized \EM{ansatz} yields the factors
$\specialfactor{a,b,c,d,p}$ as \EM{polynomials} for fixed $d$ and $p$ ($p\leq 0$).

Substituting \eqref{eq:special-ansatz} in \eqref{eq:general-recursion}, we obtain
by straightforward cancellation
the following functional equation for $\specialfactor{a,b,c,d,p}$:
\begin{multline}
\label{eq:special-recursion}
\pas{a-1}\cdot\specialfactor{a, b, c, d, p}\cdot\specialfactor{a-2, b, c, d, -1 + p} = \\
\pas{a + b - 1}\cdot\specialfactor{a-1, b, c, d, -1 + p}\cdot\specialfactor{a-1, b, c, d, p} - \\
b\cdot\specialfactor{a-1, -1 + b, 1 + c, d, p}\cdot\specialfactor{a-1, 1 + b, -1 + c, d, -1 + p}.\end{multline}
(Note that the coefficients of this functional equation do not contain the variable $c$.)

Now let $d>0$ be arbitrary, but fixed. Observe that
$$
\generalfactor{a,b,c,d,p} = \specialfactor{a, b, c, d, p} \equiv 1\text{ for all } p\leq -d
$$
(since $\evencount{a,b,c,d,p} = \macmahon{a,b,c}$ for $p\leq -d$; see \figref{d1pminus} for
an illustration): So we found our $p_0=-d$ for which $\generalfactor{a,b,c,d,p_0}$
can be derived easily.

Morevover, we (trivially) have
\begin{equation}
\label{eq:a=0}
\evencount{0,b,c,d,p}\equiv 1,
\end{equation}
hence we have
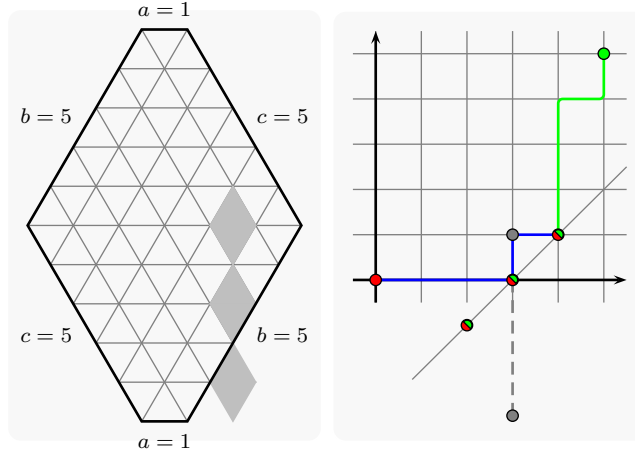
$$
\specialfactor{0, b, c, d, p} = \pas{\prod_{k=0}^{d+p-1}
\frac{\pochhammer{b+c-2d+2k+2}{2d-2-3k}}{\pochhammer{c-k}{-d-p+1+2k}}}.
$$

\psfigurescommented{
\psset{unit=0.6cm}
\begin{pspicture}(-2.95,-0.45)(3.95,9.1103)
\pspolygon[linecolor=white,fillstyle=solid,fillcolor=backgroundgray,linearc=0.3](-2.95,-0.45)(3.95,-0.45)(3.95,9.1103)(-2.95,9.1103)

\psset{linewidth=0.5pt,linecolor=gray,linestyle=solid,fillstyle=none}
\psline(-2.5,4.3301)(0,8.6603)
\psline(-2,3.4641)(1,8.6603)
\psline(-1.5,2.5981)(1.5,7.7942)
\psline(-1,1.7321)(2,6.9282)
\psline(-0.5,0.86603)(2.5,6.0622)
\psline(0,0)(3,5.1962)
\psline(1,0)(3.5,4.3301)
\psline(0,0)(-2.5,4.3301)
\psline(1,0)(-2,5.1962)
\psline(1.5,0.86603)(-1.5,6.0622)
\psline(2,1.7321)(-1,6.9282)
\psline(2.5,2.5981)(-0.5,7.7942)
\psline(3,3.4641)(0,8.6603)
\psline(3.5,4.3301)(1,8.6603)
\psline(1,0)(0,0)
\psline(1.5,0.86603)(-0.5,0.86603)
\psline(2,1.7321)(-1,1.7321)
\psline(2.5,2.5981)(-1.5,2.5981)
\psline(3,3.4641)(-2,3.4641)
\psline(3.5,4.3301)(-2.5,4.3301)
\psline(3,5.1962)(-2,5.1962)
\psline(2.5,6.0622)(-1.5,6.0622)
\psline(2,6.9282)(-1,6.9282)
\psline(1.5,7.7942)(-0.5,7.7942)
\psline(1,8.6603)(0,8.6603)
\psset{fillstyle=solid,fillcolor=lightgray,linecolor=lightgray}
\pspolygon(2,0)(2.5,0.86603)(1.5,0.86603)
\pspolygon(1.5,0.86603)(2.5,0.86603)(2,1.7321)
\pspolygon(2,1.7321)(2.5,2.5981)(1.5,2.5981)
\pspolygon(1.5,2.5981)(2.5,2.5981)(2,3.4641)
\pspolygon(2,3.4641)(2.5,4.3301)(1.5,4.3301)
\pspolygon(1.5,4.3301)(2.5,4.3301)(2,5.1962)
\psset{linewidth=1pt,linecolor=black,linestyle=solid,fillstyle=none}
\pspolygon[fillstyle=none](0,0)(1,0)(3.5,4.3301)(1,8.6603)(0,8.6603)(-2.5,4.3301)
\uput[-90.0](0.5,0){\tiny $a=1$}
\uput[-30.0](2.25,2.1651){\tiny $b=5$}
\uput[30.0](2.25,6.4952){\tiny $c=5$}
\uput[90.0](0.5,8.6603){\tiny $a=1$}
\uput[150.0](-1.25,6.4952){\tiny $b=5$}
\uput[-150.0](-1.25,2.1651){\tiny $c=5$}
\end{pspicture} 
\psset{unit=0.6cm}
\begin{pspicture}(-0.95,-3.575)(5.95,5.95)
\pspolygon[linecolor=white,fillstyle=solid,fillcolor=backgroundgray,linearc=0.3](-0.95,-3.575)(5.95,-3.575)(5.95,5.95)(-0.95,5.95)

\psset{linewidth=0.5pt,linecolor=gray,linestyle=solid,fillstyle=none}
\psline(0,-0.5)(0,5.5)
\psline(1,-0.5)(1,5.5)
\psline(2,-0.5)(2,5.5)
\psline(3,-0.5)(3,5.5)
\psline(4,-0.5)(4,5.5)
\psline(5,-0.5)(5,5.5)
\psline(-0.5,1)(5.5,1)
\psline(-0.5,2)(5.5,2)
\psline(-0.5,3)(5.5,3)
\psline(-0.5,4)(5.5,4)
\psline(-0.5,5)(5.5,5)
\psset{linewidth=1pt,linecolor=black,linestyle=solid,fillstyle=none}
\psline{->}(-0.5,0)(5.5,0)
\psline{->}(0,-0.5)(0,5.5)
\psset{linewidth=0.1,linecolor=blue,linearc=0.1}
\psset{linewidth=0.5pt,linecolor=gray,linestyle=solid,fillstyle=none}
\psline(0.8,-2.2)(5.5,2.5)
\psset{linewidth=1pt,linecolor=black,linestyle=solid,fillstyle=none}
\psset{linecolor=blue}
\psline(0,0)(3,0)(3,1)(4,1)
\psset{linestyle=dashed,linecolor=gray}
\psline(3,-3)(3,0)
\psset{linestyle=solid}
\pscircle[fillstyle=solid,linecolor=black,linewidth=0.5pt,fillcolor=red,fillcolor=gray](3,-3){0.125}
\psset{linecolor=green}
\psline(4,1)(4,4)(5,4)(5,5)
\pscircle[fillstyle=solid,linecolor=black,linewidth=0.5pt,fillcolor=red,fillcolor=gray](3,1){0.125}
\pscircle[fillstyle=solid,linecolor=black,linewidth=0.5pt,fillcolor=red,](0,0){0.125}
\pscircle[fillstyle=solid,linecolor=black,linewidth=0.5pt,fillcolor=red,](2,-1){0.125}
\pscircle[fillstyle=solid,linecolor=black,linewidth=0.5pt,fillcolor=red,](3,0){0.125}
\pscircle[fillstyle=solid,linecolor=black,linewidth=0.5pt,fillcolor=red,](4,1){0.125}
\pscircle[fillstyle=solid,linecolor=black,linewidth=0.5pt,fillcolor=red,fillcolor=green](5,5){0.125}
\pswedge[fillstyle=solid,linecolor=black,linewidth=0.5pt,fillcolor=green](2,-1){0.125}{-45}{135}
\pswedge[fillstyle=solid,linecolor=black,linewidth=0.5pt,fillcolor=green](3,0){0.125}{-45}{135}
\pswedge[fillstyle=solid,linecolor=black,linewidth=0.5pt,fillcolor=green](4,1){0.125}{-45}{135}
\end{pspicture} 
}{
The hexagon with side lengths $\pas{a,b,c}=\pas{1,5,5}$ and even intrusion of length $d=3$ at position $p=-1$.%
}{
The left picture shows the hexagon with side lengths $\pas{a,b,c}=\pas{1,5,5}$
with an \EM{even intrusion} of length $d=3$ (marked as gray triangles) at position $p=-1$, and
the right picture shows the starting and ending points of the \nilp s which correspond
to lozenge tilings of this damaged hexagon. Note that the lattice paths \EM{with} intersections
are precisely those which
\bit
\item run from starting point $\pas{0,0}$ to $\pas{3,0}$, and continue from there
to ending point $\pas{5,5}$; the number of such lattice paths
is
$$
\binom{3}{3}\binom{7}{2},
$$
\item or run from starting point $\pas{0,0}$ to $\pas{3,1}$, \EM{without} touching
$\pas{3,0}$, then make a horizontal step to $\pas{4,1}$, and continue from there
to ending point $\pas{5,5}$;
by the \EM{reflection principle} (the reflected path is shown with dashed lines), the number of 
such lattice paths is
$$
\pas{\binom{4}{3} - \binom{4}{0}}\binom{5}{1}.
$$
\eit
So altogether, the number of \nilp s is
$$
\binom{10}{5} - \binom{3}{3}\binom{7}{2} - \pas{\binom{4}{3} - \binom{4}{0}}\binom{5}{1}.
$$
}{
fig:reflection_principle%
}

So all that is left to find is a formula which gives $\specialfactor{1, b, c, d, p}$:
By a straightforward application of the \EM{reflection principle} \cite{andre:1887} (see \figref{fig:reflection_principle}
for an illustration) we
obtain the following expression for $\evencount{1,b,c,d,p}$ for
$p\leq 0$ and $d\leq\frac{b+c+1}2$:
{\small
$$
\binom{b + c}{b} - 
\sum_{i=0}^{d+p-1}\pas{
	\binom{b + c - 2 (-p + i) - 1}{ b + 2 p - i - 1}
	\pas{\binom{2 (-p + i)}{ -2 p + i} - \binom{2 (-p + i)}{ -1 + i}}
	}.
$$
} %
We may rewrite this as follows:
\begin{equation}
\label{eq:formula-a-1}
{\binom{b + c}{b}} - \left(1 - 2 p\right) \sum_{i=0}^{d + p - 1} \frac{{\left(2 i - 2 p\right)}_{\left(i - 1\right)} {\binom{b + c - 2 i + 2 p - 1}{b - i + 2 p - 1}}}{i!}.
\end{equation}
This implies that $\specialfactor{1, b, c, d, p}$
equals the right--hand side of  \eqref{eq:formula-a-1}, divided by $\binom{b+c}b=\macmahon{1,b,c}$ and
by the product in \eqref{eq:special-ansatz}.

In order to simplify notation, set
$$
\poly{r}{a,i,x} = \specialfactor{a,b+i,c-i,d,-d+x}.
$$
Clearly, the desired formula $\specialfactor{a,b,c,d,p}$ is some function in $a,b,c,\poly{r}{0,i,x}$
and $\poly{r}{1,i,x}$: As already mentioned, numerical experiments indicate that it is,
in fact, a \EM{polynomial} for $d$ and $p\leq 0$ fixed.

\secC{Special case $p=1-d$ (or $x=1$)}
\label{sec:p=1-d}
For $x=1$ (equivalent to $p=1-d$) we claim 
\begin{multline}
\label{eq:x=1}
\specialfactor{a,b,c,d,1-d} = \\
\frac{\sum_{k=0}^{a-1}\pas{-1}^{a+k-1}\pochhammer{-a+b+k+2}{a-1}\binom{a-1}k\poly{r}{1,-a+k+1,1}}{\pas{a-1}!}
\end{multline}
for all $a,d>0$.

First, note that formula \eqref{eq:x=1} gives the correct result
(namely $\poly{r}{1,0,1}$) for $a=1$. In order to show its validity for $a>1$, we must verify that
it satisfies the functional equation \eqref{eq:special-recursion}, which
simplifies to
\begin{multline}
\label{eq:special-x=1}
\specialfactor{a, b, c, d, -d+1} = \\
\frac{
	\pas{a + b - 1}\cdot\specialfactor{a-1, b, c, d, -d+1} - 
	b\cdot\specialfactor{a-1,  b-1, c+1, d, -d+1}
}{a-1}
\end{multline}
since $\specialfactor{a,b,c,d,-d}\equiv 1$.
Now substitute \eqref{eq:x=1} in \eqref{eq:special-x=1}, multiply by $\pas{a-1}!$
and compare the coefficients of 
$\poly{r}{1,-a+k,1}$: On the right--hand side, this coefficient is
\begin{multline*}
-b \pas{-1}^{a+k-2}\pochhammer{-a+b+k+2}{a-2}\binom{a-2}{k} + \\
\pas{a+b-1}\pas{-1}^{a+k-3}\pochhammer{-a+b+k+2}{a-2}\binom{a-2}{k-1}.
\end{multline*}
Collecting the terms with factor $b$ and applying the recursion of binomial coefficients, we obtain
\begin{equation*}
\pas{-1}^{a+k-1}\pochhammer{-a+b+k+2}{a-2}\pas{b \binom{a-1}{k} + \pas{a-1}\binom{a-2}{k-1}}.
\end{equation*}
Now rewrite $\pas{a-1}\binom{a-2}{k-1} = k\binom{a-1}{k}$ to arrive at
\begin{multline*}
\pas{-1}^{a+k-1}\pas{b+k}\pochhammer{-a+b+k+2}{a-2}\binom{a-1}{k} = \\\pas{-1}^{a+k-1}\pochhammer{-a+b+k+2}{a-1}\binom{a-1}{k},
\end{multline*}
which is precisely the coefficient on the left--hand side: This proves that \eqref{eq:x=1}
does indeed satisfy the functional equation \eqref{eq:special-recursion}.

By combining \eqref{eq:special-ansatz}, \eqref{eq:macmahon} and \eqref{eq:formula-a-1}
we get
$$
\poly{r}{1,i,1} = \frac{\pas{\binom{b+c}{c} - \binom{b+c-2d+1}{c-i}}\pochhammer{b+c-2d+2}{a+2d-2}}{\binom{b+c}{c}\pochhammer ca}.
$$
Inserting this in \eqref{eq:x=1} leads to the following result:
\begin{pro}
\label{pro:p=1-d_simple}
Consider the damaged $\pas{a,b,c}$--hexagon with an even intrusion of length $d>0$ in
position $p=1-d$.

For $d\geq\frac b2 + 1$, we have
$$
\evencount{a,b,c,d,1-d} = \macmahon{a,b,c}.
$$
For $d < \frac b2 + 1$, we have
{\small
\begin{multline}
\label{eq:p=1-d_simple}
\evencount{a,b,c,d,1-d} =
\macmahon{a,b,c}\cdot\Biggl(1-\frac{\pochhammer ca}{\pas{a-1}!\pochhammer{b+c-2d+2}{a+2d-2}} \times\\
\sum_{k=0}^{a-1}
	\pas{-1}^{a+k-1}
	\binom{a-1}k
	\frac{\pochhammer{-a+b-2d+k+3}{a+2d-2}}{a+c-k-1}
\Biggr).
\end{multline}
}
\end{pro}
\begin{proof}
The first assertion is an immediate consequence of the fact that an intrusion
in position $1-d$ does not inflict any actual ``damage'' to the hexagon if $d > \frac b2 +1$, see
\figref{d1pminus}.

The second assertion follows from the above considerations, which immediately give
{\small
\begin{multline}
\label{eq:p=1-d_aux}
\evencount{a,b,c,d,1-d} =
\frac{\macmahon{a,b,c}\pochhammer{c}{a}}{\pas{a-1}!\pochhammer{b+c-1}{a-1}}\times \\
\sum_{k=0}^{a-1}
	\pas{-1}^{a+k-1}
	\binom{a-1}k
	\pochhammer{-a+b+k+2}{a-1}
	\pas{\frac{\pas{\binom{b+c}{a+c-k-1} - \binom{b+c-2d+1}{a+c-k-1}}}{\binom{b+c}{a+c-k-1}\pas{a+c-k-1}}}.
\end{multline}
} The sum in this expression is the difference of two sums, the simpler of which is
$$
S_a \defeq
\sum_{k=0}^{a-1}
	\pas{-1}^{a+k-1}
	\binom{a-1}k
	\frac{\pochhammer{-a+b+k+2}{a-1}}{a+c-k-1}.
$$
Zeilberger's algorithm \cite{Zeilberger:1991:TMOCT,zb:risc} readily gives the recursion
$$S_{a+1} = \frac{a \pas{a+b+c}}{a+c}S_{a},$$
from which we immediately obtain the following summation formula:
$$
S_a =
\frac{\pas{a-1}!\pas{c-1}!\pas{a+b+c-1}!}{\pas{a+c-1}!\pas{b+c}!}
$$
Use of this formula together with straightforward simplifications 
yields \eqref{eq:p=1-d_simple}.
\end{proof}

\secC{Very special case $x=1$ and $d=1$ (so $p=0$)}
The case $d=1$, $p=0$ is particularly simple: By the recursion for binomial coefficients
and the identity $\frac nk\binom{n-1}{k-1} = \binom nk$ (for $0<k\leq n$), we have
$$
\frac{\pas{\binom{b+c}{a+c-k-1} - \binom{b+c-1}{a+c-k-1}}}{\binom{b+c}{a+c-k-1}\pas{a+c-k-1}} = 1
$$
in \eqref{eq:p=1-d_aux}, whence the sum simplifies to
\begin{equation}
\sum_{k=0}^{a-1}(-1)^{a+k-1}(-a+b+k+2)_{a-1}\binom{a-1}k = (a-1)!
\end{equation}
(Again, this summation formula is readily found by Zeilberger's algorithm \cite{Zeilberger:1991:TMOCT,zb:risc}.)

\begin{cor}
Consider the damaged $\pas{a,b,c}$--hexagon with an even intrusion of length $d=1$ in
position $p=0$. Then we have
$$
\evencount{a,b,c,1,0} = \macmahon{a,b,c}\cdot\frac{\pochhammer ca}{\pochhammer{b+c}a} = \macmahon{a,b,c-1}.
$$
\end{cor}

\psfigurescommented{
\psset{unit=0.6cm}
\begin{pspicture}(-2.95,-0.45)(5.45,8.2442)
\pspolygon[linecolor=white,fillstyle=solid,fillcolor=backgroundgray,linearc=0.3](-2.95,-0.45)(5.45,-0.45)(5.45,8.2442)(-2.95,8.2442)

\psset{linewidth=0.5pt,linecolor=gray,linestyle=solid,fillstyle=none}
\psline(-2.5,4.3301)(-0.5,7.7942)
\psline(-2,3.4641)(0.5,7.7942)
\psline(-1.5,2.5981)(1.5,7.7942)
\psline(-1,1.7321)(2.5,7.7942)
\psline(-0.5,0.86603)(3,6.9282)
\psline(0,0)(3.5,6.0622)
\psline(1,0)(4,5.1962)
\psline(2,0)(4.5,4.3301)
\psline(3,0)(5,3.4641)
\psline(0,0)(-2.5,4.3301)
\psline(1,0)(-2,5.1962)
\psline(2,0)(-1.5,6.0622)
\psline(3,0)(-1,6.9282)
\psline(3.5,0.86603)(-0.5,7.7942)
\psline(4,1.7321)(0.5,7.7942)
\psline(4.5,2.5981)(1.5,7.7942)
\psline(5,3.4641)(2.5,7.7942)
\psline(3,0)(0,0)
\psline(3.5,0.86603)(-0.5,0.86603)
\psline(4,1.7321)(-1,1.7321)
\psline(4.5,2.5981)(-1.5,2.5981)
\psline(5,3.4641)(-2,3.4641)
\psline(4.5,4.3301)(-2.5,4.3301)
\psline(4,5.1962)(-2,5.1962)
\psline(3.5,6.0622)(-1.5,6.0622)
\psline(3,6.9282)(-1,6.9282)
\psline(2.5,7.7942)(-0.5,7.7942)
\psset{fillstyle=solid,fillcolor=lightgray,linecolor=lightgray}
\pspolygon(3,0)(3.5,0.86603)(2.5,0.86603)
\pspolygon(2.5,0.86603)(3.5,0.86603)(3,1.7321)
\psset{linewidth=1pt,linecolor=black,linestyle=solid,fillstyle=none}
\pspolygon[fillstyle=none](0,0)(3,0)(5,3.4641)(2.5,7.7942)(-0.5,7.7942)(-2.5,4.3301)
\uput[-90.0](1.5,0){\tiny $a=3$}
\uput[-30.0](4,1.7321){\tiny $b=4$}
\uput[30.0](3.75,5.6292){\tiny $c=5$}
\uput[90.0](1,7.7942){\tiny $a=3$}
\uput[150.0](-1.5,6.0622){\tiny $b=4$}
\uput[-150.0](-1.25,2.1651){\tiny $c=5$}
\pspolygon[fillstyle=solid,fillcolor=mfzarttuerkis,linecolor=black](0,0)(1,0)(0.5,0.86603)(-0.5,0.86603)
\pspolygon[fillstyle=solid,fillcolor=mfzarttuerkis,linecolor=black](3,1.7321)(3.5,0.86603)(4,1.7321)(3.5,2.5981)
\pspolygon[fillstyle=solid,fillcolor=mfzarttuerkis,linecolor=black](1,0)(2,0)(1.5,0.86603)(0.5,0.86603)
\pspolygon[fillstyle=solid,fillcolor=mfzarttuerkis,linecolor=black](3.5,2.5981)(4,1.7321)(4.5,2.5981)(4,3.4641)
\pspolygon[fillstyle=solid,fillcolor=mfzarttuerkis,linecolor=black](2,0)(3,0)(2.5,0.86603)(1.5,0.86603)
\pspolygon[fillstyle=solid,fillcolor=mfzarttuerkis,linecolor=black](4,3.4641)(4.5,2.5981)(5,3.4641)(4.5,4.3301)
\end{pspicture} 
\psset{unit=0.6cm}
\begin{pspicture}(-2.95,-0.45)(5.45,8.2442)
\pspolygon[linecolor=white,fillstyle=solid,fillcolor=backgroundgray,linearc=0.3](-2.95,-0.45)(5.45,-0.45)(5.45,8.2442)(-2.95,8.2442)

\psset{linewidth=0.5pt,linecolor=gray,linestyle=solid,fillstyle=none}
\pspolygon[fillstyle=solid,fillcolor=Apricot,linecolor=white](0,0)(3,0)(5,3.4641)(2.5,7.7942)(-0.5,7.7942)(-2.5,4.3301)
\psline[linecolor=white](-2,3.4641)(0.5,7.7942)
\psline[linecolor=white](-1.5,2.5981)(1.5,7.7942)
\psline[linecolor=white](-1,1.7321)(2.5,7.7942)
\psline[linecolor=white](-0.5,0.86603)(3,6.9282)
\psline[linecolor=white](0,0)(3.5,6.0622)
\psline[linecolor=white](1,0)(4,5.1962)
\psline[linecolor=white](2,0)(4.5,4.3301)
\psline[linecolor=white](-2,5.1962)(1,0)
\psline[linecolor=white](-1.5,6.0622)(2,0)
\psline[linecolor=white](-1,6.9282)(3,0)
\psline[linecolor=white](-0.5,7.7942)(3.5,0.86603)
\psline[linecolor=white](0.5,7.7942)(4,1.7321)
\psline[linecolor=white](1.5,7.7942)(4.5,2.5981)
\psset{fillstyle=solid,fillcolor=white,linecolor=white}
\pspolygon(3,0)(3.5,0.86603)(2.5,0.86603)
\pspolygon(2.5,0.86603)(3.5,0.86603)(3,1.7321)
\pspolygon[fillstyle=solid,fillcolor=Mahogany,linecolor=white](2,0)(3,0)(2.5,0.86603)(1.5,0.86603)
\pspolygon[fillstyle=solid,fillcolor=Tan,linecolor=white](1.5,0.86603)(2.5,0.86603)(3,1.7321)(2,1.7321)
\pspolygon[fillstyle=solid,fillcolor=Tan,linecolor=white](2,1.7321)(3,1.7321)(3.5,2.5981)(2.5,2.5981)
\pspolygon[fillstyle=solid,fillcolor=Tan,linecolor=white](2.5,2.5981)(3.5,2.5981)(4,3.4641)(3,3.4641)
\pspolygon[fillstyle=solid,fillcolor=Mahogany,linecolor=white](3,3.4641)(4,3.4641)(3.5,4.3301)(2.5,4.3301)
\pspolygon[fillstyle=solid,fillcolor=Mahogany,linecolor=white](2.5,4.3301)(3.5,4.3301)(3,5.1962)(2,5.1962)
\pspolygon[fillstyle=solid,fillcolor=Tan,linecolor=white](2,5.1962)(3,5.1962)(3.5,6.0622)(2.5,6.0622)
\pspolygon[fillstyle=solid,fillcolor=Mahogany,linecolor=white](2.5,6.0622)(3.5,6.0622)(3,6.9282)(2,6.9282)
\pspolygon[fillstyle=solid,fillcolor=Mahogany,linecolor=white](2,6.9282)(3,6.9282)(2.5,7.7942)(1.5,7.7942)
\pspolygon[fillstyle=solid,fillcolor=Mahogany,linecolor=white](1,0)(2,0)(1.5,0.86603)(0.5,0.86603)
\pspolygon[fillstyle=solid,fillcolor=Mahogany,linecolor=white](0.5,0.86603)(1.5,0.86603)(1,1.7321)(0,1.7321)
\pspolygon[fillstyle=solid,fillcolor=Tan,linecolor=white](0,1.7321)(1,1.7321)(1.5,2.5981)(0.5,2.5981)
\pspolygon[fillstyle=solid,fillcolor=Tan,linecolor=white](0.5,2.5981)(1.5,2.5981)(2,3.4641)(1,3.4641)
\pspolygon[fillstyle=solid,fillcolor=Mahogany,linecolor=white](1,3.4641)(2,3.4641)(1.5,4.3301)(0.5,4.3301)
\pspolygon[fillstyle=solid,fillcolor=Tan,linecolor=white](0.5,4.3301)(1.5,4.3301)(2,5.1962)(1,5.1962)
\pspolygon[fillstyle=solid,fillcolor=Tan,linecolor=white](1,5.1962)(2,5.1962)(2.5,6.0622)(1.5,6.0622)
\pspolygon[fillstyle=solid,fillcolor=Mahogany,linecolor=white](1.5,6.0622)(2.5,6.0622)(2,6.9282)(1,6.9282)
\pspolygon[fillstyle=solid,fillcolor=Mahogany,linecolor=white](1,6.9282)(2,6.9282)(1.5,7.7942)(0.5,7.7942)
\pspolygon[fillstyle=solid,fillcolor=Mahogany,linecolor=white](0,0)(1,0)(0.5,0.86603)(-0.5,0.86603)
\pspolygon[fillstyle=solid,fillcolor=Mahogany,linecolor=white](-0.5,0.86603)(0.5,0.86603)(0,1.7321)(-1,1.7321)
\pspolygon[fillstyle=solid,fillcolor=Mahogany,linecolor=white](-1,1.7321)(0,1.7321)(-0.5,2.5981)(-1.5,2.5981)
\pspolygon[fillstyle=solid,fillcolor=Tan,linecolor=white](-1.5,2.5981)(-0.5,2.5981)(0,3.4641)(-1,3.4641)
\pspolygon[fillstyle=solid,fillcolor=Tan,linecolor=white](-1,3.4641)(0,3.4641)(0.5,4.3301)(-0.5,4.3301)
\pspolygon[fillstyle=solid,fillcolor=Tan,linecolor=white](-0.5,4.3301)(0.5,4.3301)(1,5.1962)(0,5.1962)
\pspolygon[fillstyle=solid,fillcolor=Tan,linecolor=white](0,5.1962)(1,5.1962)(1.5,6.0622)(0.5,6.0622)
\pspolygon[fillstyle=solid,fillcolor=Mahogany,linecolor=white](0.5,6.0622)(1.5,6.0622)(1,6.9282)(0,6.9282)
\pspolygon[fillstyle=solid,fillcolor=Mahogany,linecolor=white](0,6.9282)(1,6.9282)(0.5,7.7942)(-0.5,7.7942)
\psset{linewidth=1pt,linecolor=black,linestyle=solid,fillstyle=none}
\pspolygon[fillstyle=none](0,0)(3,0)(5,3.4641)(2.5,7.7942)(-0.5,7.7942)(-2.5,4.3301)
\psset{linewidth=0.1,linecolor=blue,linearc=0.15}
\end{pspicture} %
\vskip1em
\psset{unit=0.6cm}
\begin{pspicture}(-2.45,-0.45)(5.45,7.3782)
\pspolygon[linecolor=white,fillstyle=solid,fillcolor=backgroundgray,linearc=0.3](-2.45,-0.45)(5.45,-0.45)(5.45,7.3782)(-2.45,7.3782)

\psset{linewidth=0.5pt,linecolor=gray,linestyle=solid,fillstyle=none}
\psline(-2,3.4641)(0,6.9282)
\psline(-1.5,2.5981)(1,6.9282)
\psline(-1,1.7321)(2,6.9282)
\psline(-0.5,0.86603)(3,6.9282)
\psline(0,0)(3.5,6.0622)
\psline(1,0)(4,5.1962)
\psline(2,0)(4.5,4.3301)
\psline(3,0)(5,3.4641)
\psline(0,0)(-2,3.4641)
\psline(1,0)(-1.5,4.3301)
\psline(2,0)(-1,5.1962)
\psline(3,0)(-0.5,6.0622)
\psline(3.5,0.86603)(0,6.9282)
\psline(4,1.7321)(1,6.9282)
\psline(4.5,2.5981)(2,6.9282)
\psline(5,3.4641)(3,6.9282)
\psline(3,0)(0,0)
\psline(3.5,0.86603)(-0.5,0.86603)
\psline(4,1.7321)(-1,1.7321)
\psline(4.5,2.5981)(-1.5,2.5981)
\psline(5,3.4641)(-2,3.4641)
\psline(4.5,4.3301)(-1.5,4.3301)
\psline(4,5.1962)(-1,5.1962)
\psline(3.5,6.0622)(-0.5,6.0622)
\psline(3,6.9282)(0,6.9282)
\psset{fillstyle=solid,fillcolor=lightgray,linecolor=lightgray}
\psset{linewidth=1pt,linecolor=black,linestyle=solid,fillstyle=none}
\pspolygon[fillstyle=none](0,0)(3,0)(5,3.4641)(3,6.9282)(0,6.9282)(-2,3.4641)
\uput[-90.0](1.5,0){\tiny $a=3$}
\uput[-30.0](4,1.7321){\tiny $b=4$}
\uput[30.0](4,5.1962){\tiny $c=4$}
\uput[90.0](1.5,6.9282){\tiny $a=3$}
\uput[150.0](-1,5.1962){\tiny $b=4$}
\uput[-150.0](-1,1.7321){\tiny $c=4$}
\end{pspicture} 
\psset{unit=0.6cm}
\begin{pspicture}(-2.45,-0.45)(5.45,7.3782)
\pspolygon[linecolor=white,fillstyle=solid,fillcolor=backgroundgray,linearc=0.3](-2.45,-0.45)(5.45,-0.45)(5.45,7.3782)(-2.45,7.3782)

\psset{linewidth=0.5pt,linecolor=gray,linestyle=solid,fillstyle=none}
\pspolygon[fillstyle=solid,fillcolor=Apricot,linecolor=white](0,0)(3,0)(5,3.4641)(3,6.9282)(0,6.9282)(-2,3.4641)
\psline[linecolor=white](-1.5,2.5981)(1,6.9282)
\psline[linecolor=white](-1,1.7321)(2,6.9282)
\psline[linecolor=white](-0.5,0.86603)(3,6.9282)
\psline[linecolor=white](0,0)(3.5,6.0622)
\psline[linecolor=white](1,0)(4,5.1962)
\psline[linecolor=white](2,0)(4.5,4.3301)
\psline[linecolor=white](-1.5,4.3301)(1,0)
\psline[linecolor=white](-1,5.1962)(2,0)
\psline[linecolor=white](-0.5,6.0622)(3,0)
\psline[linecolor=white](0,6.9282)(3.5,0.86603)
\psline[linecolor=white](1,6.9282)(4,1.7321)
\psline[linecolor=white](2,6.9282)(4.5,2.5981)
\psset{fillstyle=solid,fillcolor=white,linecolor=white}
\pspolygon(3,0)(3.5,0.86603)(2.5,0.86603)
\pspolygon(2.5,0.86603)(3.5,0.86603)(3,1.7321)
\pspolygon[fillstyle=solid,fillcolor=Tan,linecolor=white](2,0)(3,0)(3.5,0.86603)(2.5,0.86603)
\pspolygon[fillstyle=solid,fillcolor=Tan,linecolor=white](2.5,0.86603)(3.5,0.86603)(4,1.7321)(3,1.7321)
\pspolygon[fillstyle=solid,fillcolor=Tan,linecolor=white](3,1.7321)(4,1.7321)(4.5,2.5981)(3.5,2.5981)
\pspolygon[fillstyle=solid,fillcolor=Mahogany,linecolor=white](3.5,2.5981)(4.5,2.5981)(4,3.4641)(3,3.4641)
\pspolygon[fillstyle=solid,fillcolor=Mahogany,linecolor=white](3,3.4641)(4,3.4641)(3.5,4.3301)(2.5,4.3301)
\pspolygon[fillstyle=solid,fillcolor=Tan,linecolor=white](2.5,4.3301)(3.5,4.3301)(4,5.1962)(3,5.1962)
\pspolygon[fillstyle=solid,fillcolor=Mahogany,linecolor=white](3,5.1962)(4,5.1962)(3.5,6.0622)(2.5,6.0622)
\pspolygon[fillstyle=solid,fillcolor=Mahogany,linecolor=white](2.5,6.0622)(3.5,6.0622)(3,6.9282)(2,6.9282)
\pspolygon[fillstyle=solid,fillcolor=Mahogany,linecolor=white](1,0)(2,0)(1.5,0.86603)(0.5,0.86603)
\pspolygon[fillstyle=solid,fillcolor=Tan,linecolor=white](0.5,0.86603)(1.5,0.86603)(2,1.7321)(1,1.7321)
\pspolygon[fillstyle=solid,fillcolor=Tan,linecolor=white](1,1.7321)(2,1.7321)(2.5,2.5981)(1.5,2.5981)
\pspolygon[fillstyle=solid,fillcolor=Mahogany,linecolor=white](1.5,2.5981)(2.5,2.5981)(2,3.4641)(1,3.4641)
\pspolygon[fillstyle=solid,fillcolor=Tan,linecolor=white](1,3.4641)(2,3.4641)(2.5,4.3301)(1.5,4.3301)
\pspolygon[fillstyle=solid,fillcolor=Tan,linecolor=white](1.5,4.3301)(2.5,4.3301)(3,5.1962)(2,5.1962)
\pspolygon[fillstyle=solid,fillcolor=Mahogany,linecolor=white](2,5.1962)(3,5.1962)(2.5,6.0622)(1.5,6.0622)
\pspolygon[fillstyle=solid,fillcolor=Mahogany,linecolor=white](1.5,6.0622)(2.5,6.0622)(2,6.9282)(1,6.9282)
\pspolygon[fillstyle=solid,fillcolor=Mahogany,linecolor=white](0,0)(1,0)(0.5,0.86603)(-0.5,0.86603)
\pspolygon[fillstyle=solid,fillcolor=Mahogany,linecolor=white](-0.5,0.86603)(0.5,0.86603)(0,1.7321)(-1,1.7321)
\pspolygon[fillstyle=solid,fillcolor=Tan,linecolor=white](-1,1.7321)(0,1.7321)(0.5,2.5981)(-0.5,2.5981)
\pspolygon[fillstyle=solid,fillcolor=Tan,linecolor=white](-0.5,2.5981)(0.5,2.5981)(1,3.4641)(0,3.4641)
\pspolygon[fillstyle=solid,fillcolor=Tan,linecolor=white](0,3.4641)(1,3.4641)(1.5,4.3301)(0.5,4.3301)
\pspolygon[fillstyle=solid,fillcolor=Tan,linecolor=white](0.5,4.3301)(1.5,4.3301)(2,5.1962)(1,5.1962)
\pspolygon[fillstyle=solid,fillcolor=Mahogany,linecolor=white](1,5.1962)(2,5.1962)(1.5,6.0622)(0.5,6.0622)
\pspolygon[fillstyle=solid,fillcolor=Mahogany,linecolor=white](0.5,6.0622)(1.5,6.0622)(1,6.9282)(0,6.9282)
\psset{linewidth=1pt,linecolor=black,linestyle=solid,fillstyle=none}
\pspolygon[fillstyle=none](0,0)(3,0)(5,3.4641)(3,6.9282)(0,6.9282)(-2,3.4641)
\psset{linewidth=0.1,linecolor=blue,linearc=0.15}
\end{pspicture} 
}{
The hexagon with side lengths $\pas{a,b,c}=\pas{3,4,5}$ and even intrusion of length $d=1$ at position $0$.%
}{
The upper left picture shows the hexagon with side lengths $\pas{a,b,c}=\pas{3,4,5}$
with an \EM{even intrusion} of length $d=1$ (marked as gray triangle) at position $0$, and
the upper right picture shows a \EM{lozenge tiling} of this damaged hexagon. Note that
the intrusion implies that certain lozenges \EM{must} belong to \EM{all} tilings of the
damaged hexagon:
These \EM{forced} lozenges are drawn with blue colour in the upper left picture. But this means
that tilings of the damaged of the upper left picture are in bijection with tilings of the
(intact) hexagon with side lengths
$\pas{a,b,c}=\pas{3,4,4}$ shown in the lower left picture (the lower right picture shows the tiling
which is in bijection with the tiling from the upper right picture; the bijection simply
``removes'' the ``forced lozenges'').
}{
d1p0%
}

\bpf
The assertion is a direct consequence of the above considerations. But there is a
much simpler argument: A single intruding lozenge in position $0$ implies that
\EM{all} lozenges at the baseline of the hexagon are \EM{forced} (or, equivalently, that
\EM{all} lattice paths have to start with an \EM{upwards} step), see \figref{d1p0}.
Removing the forced lozenges gives an $\pas{a,b,c-1}$--hexagon with \EM{no} intrusion.
\epf

\secC{Special case $p=1-d$ (or $x=1$), revisited}
\label{sec:p=1-d_rev}
We may choose another \EM{ansatz}, which leads to a different formula:
\begin{pro}
\label{pro:p=1-d}
Consider the damaged $\pas{a,b,c}$--hexagon with an even intrusion of length $d>0$ in
position $p=1-d$.

For $d \leq \ceil{\frac b2}$, we have the following formula:
{\small
\begin{multline}
\label{eq:p=1-d}
\evencount{a,b,c,d,-d+1} = \macmahon{a,b,c} \\
\times
\pas{
	1 -\frac{\pochhammer{b-2d+2}{c}}{\pas{2d-2}!\pochhammer{b+1}{a+c-1}}
	\sum_{k=1}^a\pochhammer{b+c+k}{a-k}\pochhammer{k}{2d-2}\pochhammer{c}{k-1}
}.
\end{multline}
}

Alternatively, we have the following formula, valid for $b>d$
{\small
\begin{multline}
\label{eq:p=1-d_alternativ}
\evencount{a,b,c,d,1-d} = 
\frac{\pochhammer{c}{a} \macmahon{a,b,c}}{\pochhammer{b+c-2d-2}{a+2d-2}} \\
\times
\Biggl(\pochhammer{b+c-2d+2}{2d-2} \\
- \sum_{k=2}^d
	\Bigl(
	\pochhammer{b + c - 2 d + 2}{2 d - 2 k}
	\pochhammer{b - 2 k + 4}{2 k - 3}
	\pochhammer{a}{2k-3} \\
	\times\frac{-a (b-2 k+3)+b (5-4 k)-2 c k+2 c+8 k^2-20 k+13}{\pas{2k-2}!}
	\Bigr)
\Biggr),
\end{multline}
}
which has the advantage that the expression \EM{after} the first line of \eqref{eq:p=1-d_alternativ}
actually is a \EM{polynomial} in $a,b,c$
for fixed $d$.
\end{pro}
\psfigurescommented{
\psset{unit=0.6cm}
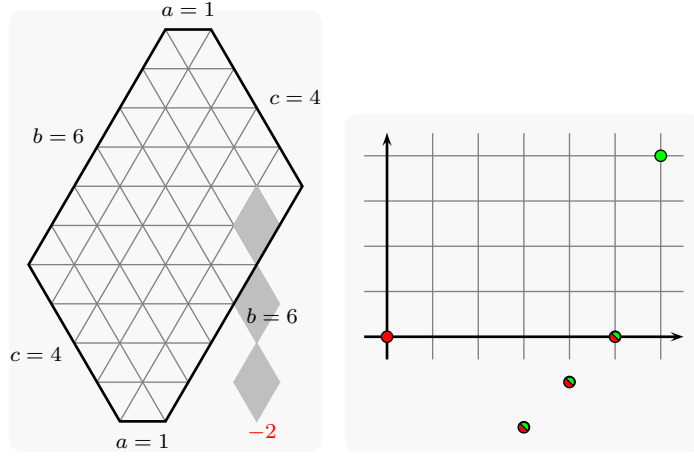
\begin{pspicture}(-2.45,-0.7)(4.45,9.1103)
\pspolygon[linecolor=white,fillstyle=solid,fillcolor=backgroundgray,linearc=0.3](-2.45,-0.7)(4.45,-0.7)(4.45,9.1103)(-2.45,9.1103)

\psset{linewidth=0.5pt,linecolor=gray,linestyle=solid,fillstyle=none}
\psline(-2,3.4641)(1,8.6603)
\psline(-1.5,2.5981)(2,8.6603)
\psline(-1,1.7321)(2.5,7.7942)
\psline(-0.5,0.86603)(3,6.9282)
\psline(0,0)(3.5,6.0622)
\psline(1,0)(4,5.1962)
\psline(0,0)(-2,3.4641)
\psline(1,0)(-1.5,4.3301)
\psline(1.5,0.86603)(-1,5.1962)
\psline(2,1.7321)(-0.5,6.0622)
\psline(2.5,2.5981)(0,6.9282)
\psline(3,3.4641)(0.5,7.7942)
\psline(3.5,4.3301)(1,8.6603)
\psline(4,5.1962)(2,8.6603)
\psline(1,0)(0,0)
\psline(1.5,0.86603)(-0.5,0.86603)
\psline(2,1.7321)(-1,1.7321)
\psline(2.5,2.5981)(-1.5,2.5981)
\psline(3,3.4641)(-2,3.4641)
\psline(3.5,4.3301)(-1.5,4.3301)
\psline(4,5.1962)(-1,5.1962)
\psline(3.5,6.0622)(-0.5,6.0622)
\psline(3,6.9282)(0,6.9282)
\psline(2.5,7.7942)(0.5,7.7942)
\psline(2,8.6603)(1,8.6603)
\psset{fillstyle=solid,fillcolor=lightgray,linecolor=lightgray}
\pspolygon(3,0)(3.5,0.86603)(2.5,0.86603)
\pspolygon(2.5,0.86603)(3.5,0.86603)(3,1.7321)
\pspolygon(3,1.7321)(3.5,2.5981)(2.5,2.5981)
\pspolygon(2.5,2.5981)(3.5,2.5981)(3,3.4641)
\pspolygon(3,3.4641)(3.5,4.3301)(2.5,4.3301)
\pspolygon(2.5,4.3301)(3.5,4.3301)(3,5.1962)
\rput(3,-0.25){ {\red $\scriptstyle -2$}}
\psset{linewidth=1pt,linecolor=black,linestyle=solid,fillstyle=none}
\pspolygon[fillstyle=none](0,0)(1,0)(4,5.1962)(2,8.6603)(1,8.6603)(-2,3.4641)
\uput[-90.0](0.5,0){\tiny $a=1$}
\uput[-30.0](2.5,2.5981){\tiny $b=6$}
\uput[30.0](3,6.9282){\tiny $c=4$}
\uput[90.0](1.5,8.6603){\tiny $a=1$}
\uput[150.0](-0.5,6.0622){\tiny $b=6$}
\uput[-150.0](-1,1.7321){\tiny $c=4$}
\end{pspicture} 
\psset{unit=0.6cm}
\begin{pspicture}(-0.95,-2.575)(6.95,4.95)
\pspolygon[linecolor=white,fillstyle=solid,fillcolor=backgroundgray,linearc=0.3](-0.95,-2.575)(6.95,-2.575)(6.95,4.95)(-0.95,4.95)

\psset{linewidth=0.5pt,linecolor=gray,linestyle=solid,fillstyle=none}
\psline(0,-0.5)(0,4.5)
\psline(1,-0.5)(1,4.5)
\psline(2,-0.5)(2,4.5)
\psline(3,-0.5)(3,4.5)
\psline(4,-0.5)(4,4.5)
\psline(5,-0.5)(5,4.5)
\psline(6,-0.5)(6,4.5)
\psline(-0.5,1)(6.5,1)
\psline(-0.5,2)(6.5,2)
\psline(-0.5,3)(6.5,3)
\psline(-0.5,4)(6.5,4)
\psset{linewidth=1pt,linecolor=black,linestyle=solid,fillstyle=none}
\psline{->}(-0.5,0)(6.5,0)
\psline{->}(0,-0.5)(0,4.5)
\psset{linewidth=0.1,linecolor=blue,linearc=0.1}
\pscircle[fillstyle=solid,linecolor=black,linewidth=0.5pt,fillcolor=red,](0,0){0.125}
\pscircle[fillstyle=solid,linecolor=black,linewidth=0.5pt,fillcolor=red,](3,-2){0.125}
\pscircle[fillstyle=solid,linecolor=black,linewidth=0.5pt,fillcolor=red,](4,-1){0.125}
\pscircle[fillstyle=solid,linecolor=black,linewidth=0.5pt,fillcolor=red,](5,0){0.125}
\pscircle[fillstyle=solid,linecolor=black,linewidth=0.5pt,fillcolor=red,fillcolor=green](6,4){0.125}
\pswedge[fillstyle=solid,linecolor=black,linewidth=0.5pt,fillcolor=green](3,-2){0.125}{-45}{135}
\pswedge[fillstyle=solid,linecolor=black,linewidth=0.5pt,fillcolor=green](4,-1){0.125}{-45}{135}
\pswedge[fillstyle=solid,linecolor=black,linewidth=0.5pt,fillcolor=green](5,0){0.125}{-45}{135}
\end{pspicture} 
}{
Hexagon with side lengths $\pas{a,b,c}=\pas{1,6,4}$ and intrusion of length $d=3$ at position $p=-2$.%
}{
The left picture illustrates the situation of \proref{p=1-d} for $a=1$, with
$d=3$ and $p=-d+1=-2$. The right
picture is the ``translation'' of this situation to the language of \nilp s: Observe that the
lattice paths \EM{with intersections} are \EM{precisely} the ones
\bit
\item which reach point $\pas{2d-1,0} = \pas{5,0}$ by five horizontal
steps from the origin $\pas{0,0}$ (and this is the \EM{only} way
to achieve this!), 
\item and then continue from $\pas{5,0}$ in an arbitrary way to the endpoint
$\pas{b,c} = \pas{6,4}$,
\eit
so the number of \nilp s in this situation is
$$
\binom{b+c}{c} - \binom{b-2d+1+c}{c} = \binom{10}{4} - \binom{5}{4}.
$$
}{
a1d3p-2%
}
\bpf
We make the \EM{ansatz}
{\small
\begin{equation}
\label{eq:p=1-dansatz}
\evencount{a,b,c,d,1-d} = \macmahon{a,b,c}\cdot
\pas{1-\frac{\pochhammer{b-2d+2}{c}}{\pas{2d-2}!\pochhammer{b+1}{a+c-1}}
\cdot f\of{a,b,c,d}}.
\end{equation}
}

Clearly, for proving \eqref{eq:p=1-d} we have to show
\begin{equation}
\label{eq:f-ansatz}
f\of{a,b,c,d} = \sum_{k=1}^a\pochhammer{b+c+k}{a-k}\pochhammer{k}{2d-2}\pochhammer{c}{k-1}.
\end{equation}
We shall achieve this by induction on $a$. 

For $a=0$, we have $\evencount{0,b,c,d,1-d} = \macmahon{0,b,c} = 1$, and the sum
in \eqref{eq:f-ansatz} is indeed zero.

For $a=1$, it is easy to see that
$\evencount{1,b,c,d,-d+1}$ is equal
to the number of lattice paths starting in $\pas{0,0}$ and ending in $\pas{b,c}$ which \EM{do not}
pass through the lattice point $\pas{2d-1,0}$ (see \figref{a1d3p-2}). This number is 
$$
\binom{b+c}c-\binom{b+c-2d+1}{c},
$$
which equals
$$
\macmahon{1,b,c}\cdot\pas{1 - \frac{\pochhammer{b-2d+2}{c}}{\pochhammer{b+1}{c}}}
=
\frac{\pochhammer{b+1}{c}}{c!}\cdot\pas{1 - \frac{\pochhammer{b-2d+2}{c}}{\pochhammer{b+1}{c}}}.
$$
From this we immediately obtain that \eqref{eq:f-ansatz} is true also for $a=1$:
$$
\sum_{k=1}^1\pochhammer{b+c+k}{1-k}\pochhammer{k}{2d-2}\pochhammer{c}{k-1} = \pas{2d-2}!.
$$

Since $\evencount{a,b,c,d,-d} = \macmahon{a,b,c}$, Dodgson's condensation \eqref{eq:dodgson}
amounts to
{\small
\begin{align*}
\evencount{a,b,c,d,1-d}\cdot\macmahon{a-2,b,c}
&=
\macmahon{a-1,b,c}\cdot\evencount{a-1,b,c,d,1-d} \\
&-
\macmahon{a-1,b+1,c-1}\cdot\evencount{a-1,b-1,c+1,d,1-d}
\end{align*}
}
for $a>1$, and substituting our \EM{ansatz} \eqref{eq:p=1-dansatz} for $\evencount{a,b,c,d,1-d}$
gives (after straightforward
cancellations) the following recursion (in $a$) for $f\of{a,b,c,d}$:
\begin{multline}
\label{eq:f-recursion}
\pas{a-1}f\of{a,b,c,d} = \\
\pas{a+b-1}\pas{a+c-1}f\of{a-1,b,c,d} \\
- c\pas{b-2d+1}f\of{a-1,b-1,c+1,d}.
\end{multline}

So what is left to prove is that
$$
\sum_{k=1}^a\pochhammer{b+c+k}{a-k}\pochhammer{k}{2d-2}\pochhammer{c}{k-1}
$$
actually obeys the recursion \eqref{eq:f-recursion}. Using the elementary identity
$$
\pas{a + b - 1}\pas{ a + c - 1} - b\, c = \pas{a - 1}\pas{a + b + c - 1}
$$
we may rewrite \eqref{eq:f-recursion} equivalently as
\begin{multline}
\label{eq:f-recursion2}
\pas{a-1}f\of{a,b,c,d} =
\pas{a - 1}\pas{a + b + c - 1} f\of{a - 1, b, c, d} \\
+ b\, c \pas{f\of{a - 1, b, c, d} - f\of{a - 1, b - 1, c + 1, d}} \\
+ c \pas{2 d - 1} f\of{a - 1, b - 1, c + 1, d}.
\end{multline}
Now observe that
$$
(a - 1) (a + b + c - 1) \pochhammer{b + c + k}{ a - k - 1}\pochhammer{c}{k - 1} \pochhammer{k}{-2 + 2 d} 
$$
equals
$$
\pas{a-1} \pochhammer{b + c + k}{  a - k} \pochhammer{c}{k-1} \pochhammer{k}{2 d-2} ,
$$
which is $\pas{a-1}$ times the $k$--th summand of $f\of{a,b,c,d}$.
Hence we need to show that the summand for $k=a$ in $\pas{a-1}f\of{a,b,c,d}$, 
\begin{equation}
\label{eq:p=d-1_summanda}
\pas{a-1} \pochhammer{c}{a-1} \pochhammer{a}{-2 + 2 d},
\end{equation}
is equal to the last two summands of the right--hand side in \eqref{eq:f-recursion2},
which can be simplified to 
{\small
\begin{multline}
\label{eq:p=1-d_zb}
\sum_{k=1}^{a-1}\Biggl(
\pochhammer{b+c+k}{a-k-1}\pochhammer{c}{k-1} \pochhammer{k}{2 d-2} \\
\times\left(b c
   \left(1-\frac{c+k-1}{c}\right)+(2 d-1)
   (c+k-1)\right)
\Biggr).
\end{multline}
} Now Zeilberger's algorithm \cite{Zeilberger:1991:TMOCT,zb:risc} shows that
\eqref{eq:p=1-d_zb} evaluates to
\eqref{eq:p=d-1_summanda} and thus concludes the  proof of \eqref{eq:p=1-d}.

It is Zeilberger's algorithm \cite{Zeilberger:1991:TMOCT,zb:risc}, again, which 
also gives a recursion for $f\of{a,b,c,d}$
in $d$, namely
\begin{multline*}
(b-2 d+2) (b-2 d+3) f(a,b,c,d)= \\
2 (d-1) (2 d-3)
   (b+c-2 d+2) (b+c-2 d+3) f(a,b,c,d-1)+ \\
   (a+c-1)
   (a+2 d-4) (c)_{a-1} (a)_{2 d-4}\times \\
    \left(-a (b-2
   d+3)+b (5-4 d)-2 c d+2 c+8 d^2-20 d+13\right),
\end{multline*}
which by iteration leads to an alternative expression for $f$, valid for $b>d$:
\begin{multline*}
f\of{a,b,c,d} =\frac{\pas{2d-2}!}{\pochhammer{b-2d+2}{2d-1}}
\Biggl(
\pochhammer{b+c-2d+2}{a+2d-2} \\
+\pochhammer{c}{a}
\sum_{k=1}^d
	\Bigl(
	\pochhammer{b + c - 2 d + 2}{2 d - 2 k}
	\pochhammer{b - 2 k + 4}{2 k - 3}
	\pochhammer{a}{2k-3} \\
	\frac{-a (b-2 k+3)+b (5-4 k)-2 c k+2 c+8 k^2-20 k+13}{\pas{2k-2}!}
	\Bigr)
\Biggr).
\end{multline*}
Inserting this alternative expression in our ansatz \eqref{eq:p=1-dansatz}
gives \eqref{eq:p=1-d_alternativ} (after some straightforward cancellations and simplifications).
\epf

\secB{A modified ansatz for $0\leq p \leq a$}
For $0\leq p \leq a$, $a\geq 0$, $b>d>0$ and $c>d+p$, we define three products:
\begin{align}
\bisp{a,b,c,d,p} &\defeq 
\pas{\prod_{i=0}^{p-1} \frac{i!\pas{b+c-d+i}!}{\pas{b-d+i}!\pas{a+c-p+i}!}}, 
\label{eq:P0..p-1}\\
\pbisa{a,b,c,d,p} &\defeq
\pas{\prod_{i=p}^{a-1} \frac{i!\pas{b+c-d+i}!}{\pas{b+i}!\pas{c-d-p+i}!}}, 
\label{eq:Pp..a-1}\\
\bisd{a,b,c,d,p} &\defeq
\pas{\prod_{i=0}^{d-1} \frac{\pochhammer{a-p+1+i}p}{\pas{p+i}!\pochhammer{b+c-2d+1+i}{i}}}.\label{eq:Pd..a-1}
\end{align}
(Note that by the inequalities constraining the integers $a,b,c,d$ and $p$, these
products are well--defined: There is no factor $z!$ for $z<0$.)

We define
$$
\pfactor{a,b,c,d,p} \defeq \bisp{a,b,c,d,p} \cdot \pbisa{a,b,c,d,p} \cdot \bisd{a,b,c,d,p}
$$
and make the modified \EM{ansatz}
\begin{equation}
\label{eq:ansatz}
\evencount{a, b, c, d, p} = \pfactor{a,b,c,d,p}\cdot\qpoly{a, b, c, d, p}.
\end{equation}
Inserting this modified ansatz
in the condensation recursion \eqref{eq:dodgson}
gives (after a lot of straightforward cancellations) the following functional
equation for
$\qpoly{a,b,c,d,p}$, valid for $1\leq p \leq a$, $a\geq 0$, $b>d>0$ and $c>d+p$:
\begin{multline}
\qpoly{a,b,c,d,p}\cdot\qpoly{a-2,b,c,d,p-1}\cdot\pas{a+b+c-d-1}\cdot\pas{a+d-1} = \\
\qpoly{a-1,b,c,d,p}\cdot\qpoly{a-1,b,c,d,p-1}\cdot\pas{a+c-1}\cdot\pas{a+b-1} \\
-  \qpoly{a-1,b-1,c+1,d,p}\cdot\qpoly{a-1,b+1,c-1,d,p-1}\cdot\pas{c-d}\cdot\pas{b-d}.
\label{eq:dogson-cancelled}
\end{multline}

The following assertion shows that this modified ansatz \eqref{eq:ansatz}
makes sense:
\begin{pro}
\label{pro:d=1}
For $d=1$, the function $\qpoly{a,b,c,1,p}$ is a simple constant:
\begin{equation}
\label{eq:d=1}
\qpoly{a,b,c,1,p}\equiv 1 \text{ for all } 0\leq p \leq a.
\end{equation}
\end{pro}
\bpf
We shall prove \eqref{eq:d=1}
by induction on $p$:
For $p=0$, we simply have
$$
\evencount{a,b,c,1,0} = \evencount{a,b,c-1,0,0} = \macmahon{a,b,c-1},
$$
see \figref{d1p0}. Moreover, it is obvious that
$$
\pfactor{a,b,c,1,0} = \pbisa{a,b,c,1,0} = \macmahon{a,b,c-1},
$$
whence $\qpoly{a,b,c,1,0} = 1$ and (by symmetry \eqref{eq:symmetry}) $\qpoly{a,b,c,1,a} = 1$.

So assume \eqref{eq:d=1} holds for $p - 1$.
By the induction hypothesis (on $p$),
\eqref{eq:dogson-cancelled} simplifies to
{\small
\begin{multline*}
\qpoly{a,b,c,1,p} = \\
\frac{\qpoly{a-1,b,c,1,p}\cdot\pas{a+c-1}\cdot\pas{a+b-1} - \qpoly{a-1,b-1,c+1,1,p}\pas{c-1}\cdot\pas{b-1}}{\pas{a+b+c-2}\cdot\pas{a}}.
\end{multline*}
}

From this, the assertion follows by induction on $a\geq p$: Simply observe that
$$
\frac{\pas{a+c-1}\cdot\pas{a+b-1} - \pas{c-1}\cdot\pas{b-1}}{\pas{a+b+c-2}\cdot\pas{a}} = 1
$$
and $\qpoly{p,b,c,1,p} = \qpoly{p,c,b,1,0} = 1.$
\epf

Note that the proof of \proref{d=1} relied on one crucial ingredient, namely
the (very simple) formula for $\qpoly{a,b,c,1,0}$ (easily obtained by the
simple formula for $\evencount{a,b,c,1,0}$):
It served
\bit
\item as the base case for the induction on $p$
\item \EM{and} as the base case for the induction on $a\geq p$ (via the symmetry
	$\evencount{p,b,c,d,p} = \evencount{p,c,b,d,0}$).
\eit
Of course, we cannot expect that $\qpoly{a,b,c,d,p}$ is given by a simple formula
for $d>1$. But numerical experiments indicate that for $d$ fixed, $\qpoly{a,b,c,d,p}$
is a \EM{polynomial} in $a,b,c,p$. So if we can somehow \EM{guess} this polynomial
\EM{and} are able to show that
\bit
\item $\evencount{a,b,c,d,0} = \pfactor{a,b,c,d,0}\cdot\qpoly{a,b,c,d,0}$
\item and \eqref{eq:dogson-cancelled} is, in fact, a \EM{polynomial identity},
\eit
then we would have proved the corresponding formula.

Assuming $b\geq 2d$, the values for $p\leq a$ are partitioned in three intervals
of ``different quality'':
\bit
\item $p\in\brk{0,a}$ (this is the interval considered in \proref{d=1}, for which
	we presented our modified \EM{ansatz}),
\item $p\in\brk{-d+1,-1}$,
\item and $p\leq-d$: It is obvious that the intrusion does no damage to the hexagon
	at all in this case, whence $\evencount{a,b,c,d,p} = \macmahon{a,b,c}$ for $p\leq -d$.
\eit

So the case $p=1-d$ (which we already considered in sections \ref{sec:p=1-d}
and \ref{sec:p=1-d_rev}) would serve as base case for the interval $\brk{1-d,-1}$
in the same sense as $p=0$ served as base case for the interval $\brk{0,a}$ in the proof
of \proref{d=1}. So in principle, we could work with our ``specialized'' ansatz from
$p=1-d$ till $p=0$, and then continue with our ``modified'' ansatz: However, the
formulas quickly become rather unwieldy for $d>1$. So for now, we conclude this line
of investigations
with the following conjecture:

\begin{con}
\label{con:1}
The number $\evencount{a,b,c,d,p}$ of lozenge tilings of a damaged hexagon with side lengths $a,b,c$ and
vertical intrusion of depth $d$ at \EM{even} position $p$ with $0\leq p \leq a$ equals
\begin{multline}
\prod_{i=0}^{d-1}\frac{\pochhammer{a-p+i+1}p}{\pochhammer{b+c-d-i}{d-i-1}\pas{p+i}!}
\times
\prod_{i=0}^{p-1}\frac{i!\pas{b+c-d+i}!}{\pas{b-d+i}!\pas{a+c-i-1}!} \\
\times
\prod_{i=p}^{a-1}\frac{i!\pas{b+c-d+i}!}{\pas{b+i}!\pas{a+c-d-i-1}!}
\times
\qpoly{p,c,b,d,p},
\label{eq:evend-general}
\end{multline}
where for \EM{fixed} $d$ the factor $\qpoly{p,c,b,d,p}$ is a \EM{polynomial}
in the variables $a,b,c$ and $p$.
The coefficient of monomial $b^i c^j$ in $\qpoly{p,c,b,d,p}$ is
a polynomial in $a$ and $p$ whose degree with respect to $a$ is $\leq g-j$,
and whose degree with respect to $p$ is $\leq 2g-i-j$. For instance, in Proposition~\ref{pro:d=1}
we showed $\qpoly{a,b,c,1,p} = 1$. Numerical experiments indicate that
$$
\qpoly{a,b,c,2,p} = b\cdot\pas{a-p+1}+c\cdot\pas{p+1} + 2\pas{a p - p^2 - 1}
$$
and 
\begin{multline*}
\qpoly{a,b,c,3,p} = 
98 a^{3} + 621 a^{2} + 1243 a + \left(a + 3\right) \left(2 a + 5\right) \left(125 a + 250\right) - \\\left(a + 3\right) \left(4 a + 13\right) \left(75 a + 150\right) + \left(a + 3\right) \left(375 a + 750\right) + 786.
\end{multline*}
\end{con}
A brute--force computer search yields the polynomials $\qpoly{a,b,c,d,p}$ for $d$ up to $5$:
\EM{Mathematica} shows that all these
formulae factor nicely for $a=2p$, in accordance with Byun's formula \eqref{eq:byun-even}.

In order to specify and prove Conjecture~\ref{con:1}, we need to find the general
formula giving $\qpoly{a,b,c,d,p}$: We hope to find this formula in future work.

\secB{Another ``brute force'' approach}
Consider the matrix whose determinant gives MacMahon's formula:
\begin{lem}
\label{lem:LU}
For $a,b,c\in\N$, define the following $\pas{a\times a}$--matrices $M$, $L$, $T$, $D$ and $U$ with
$\pas{i,j}$--entries 
\begin{align*}
M_{i,j} &\defeq \binom{b+c}{b+i-j},\\
L_{i,j} &\defeq \pas{-1}^{i+j} \binom{i-1}{j-1} \frac{\pochhammer{c}{i-j}}{\pochhammer{b+j}{i-j}},\\
T_{i,j} &\defeq \pas{-1}^{i+j} \binom{j-1}{i-1} \frac{\pochhammer{b}{j-i}}{\pochhammer{c+i}{j-i}},\\
D_{i,j} &\defeq \Iverson{i=j}\cdot\frac{\pas{b+i-1}!\pas{c+i-1}!}{\pas{b+c+i-1}!\pas{i-1}!},\\
U_{i,j} &\defeq \pochhammer{-i+j+1}{i-1} \frac{b!\pas{b+c+i-1}!}{\pas{b+i-1}! \pas{c+j-1}! \pas{b+i-j}!}.\end{align*}
Note that $M$ is the matrix corresponding to MacMahon's formula
(i.e., $\det M = \macmahon{a,b,c}$),
$L$ is a lower triangular matrix with entries $1$ on the main diagonal,
$T$ is the transpose of $L$ with variables $b$ and $c$ swapped,
$D$ is a diagonal matrix (Iverson's bracket $\Iverson{A}$ is $1$ if assertion $A$ is true, else $0$),
and $U$ is an upper triangular matrix.

Then we have
\begin{equation}
\label{eq:triangulise_M}
U = L\cdot M
\end{equation}
and 
\begin{equation}
\label{eq:invert_M}
M^{-1} = T\cdot D\cdot L.
\end{equation}
Moreover, the $\pas{i,j}$--entry of the inverse $M^{-1}$ is
\begin{multline}
\label{eq:M^-1ij}
M^{-1}_{i,j} = 
\pas{-1}^{i+j}\pas{b+j-1}!\pas{c+i-1}!\\
\times\sum_{k=1}^a \binom{k-1}{i-1}\binom{k-1}{j-1}
\frac{\pochhammer{b}{k-i}\pochhammer{c}{k-j}}{\pas{k-1}!\pas{b+c+k-1}!}
\end{multline}
(Note that the sum in \eqref{eq:M^-1ij} actually starts at $k=\max\of{i,j}$: All other summands
are zero due to the binomial coefficients.)
\end{lem}
\begin{rem}
As an easy consequence of \eqref{eq:triangulise_M}, we have
$$
\det M = \det U = 
\prod_{i=0}^{a-1} i! \frac{\pas{b+c+i}!}{\pas{b+i}! \pas{c+i}!},
$$
which is MacMahon's formula \eqref{eq:macmahon}.
\end{rem}

Let us call the matrix $M$ in Lemma~\ref{lem:LU} \EM{MacMahon's matrix}. Consider the
natural decomposition of the matrix $Q$ underlying the determinant giving $\evencount{a,b,c,d,p}$ (i.e.,
$\det Q = \evencount{a,b,c,d,p}$, see Example~\ref{ex:hexc_gf})
into $4$ submatrices $Q_1$, $Q_2$ $Q_3$ and $Q_4$,
$$
Q=
\begin{pmatrix}
Q_2 & Q_1 \\
Q_3 & Q_4
\end{pmatrix},
$$
where
\bit
\item $Q_1$ is the submatrix of the first $a$ rows and $d$ last columns of $Q$,
\item $Q_2$ is the submatrix of the first $a$ rows and $a$ first columns of $Q$,
\item $Q_3$ is the submatrix of the last $d$ rows and $a$ first columns of $Q$,
\item $Q_4$ is the submatrix of the last $d$ rows and $d$ last columns of $Q$.
\eit
Note that $Q_2$ is MacMahon's matrix (i.e., matrix $M$ in Lemma~\ref{lem:LU}). All the $\pas{i,j}$--entries of these submatrices are
binomial coefficients:
\begin{align}
\pas{Q_1}_{i,j} &= \binom{2j-1}{-i+j+p} \label{eq:Q1ij} \\
\pas{Q_2}_{i,j} &= \binom{b+c}{c-i+j} \label{eq:Q2ij} \\
\pas{Q_3}_{i,j} &= \binom{b+c-2i+1}{c-i+j-p} \label{eq:Q3ij} \\
\pas{Q_4}_{i,j} &= \binom{2\pas{j-i}}{j-i} \label{eq:Q4ij}
\end{align}
Denote by $\eye$ and $\nought$ the identity matrix and the zero matrix, respectively,
with the ``appropriate'' dimensions, and observe
$$
\begin{pmatrix}
Q_2^{-1} & \nought \\
\nought & \eye
\end{pmatrix}
\cdot
\begin{pmatrix}
Q_2 & Q_1 \\
Q_3 & Q_4
\end{pmatrix}
=
\begin{pmatrix}
\eye & Q_2^{-1}\cdot Q_1 \\
Q_3 & Q4
\end{pmatrix}
$$
Combining \eqref{eq:M^-1ij} and \eqref{eq:Q1ij}, we see that the $\pas{i,j}$--entry of $Q_2^{-1}\cdot Q_1$ is
\begin{multline}
\label{eq:Q2^-1Q1ij}
\pas{Q_2^{-1}\cdot Q_1}_{i,j} = \\
\sum_{l=1}^a
	\pas{-1}^{i+l}\pas{b+l-1}!\pas{c+i-1}!\binom{2j-1}{l+j-p-1} \\
    \sum_{k=1}^a
    	\binom{k-1}{i-1}\binom{k-1}{l-1}
		\frac{
			\pochhammer{b}{k-i}\pochhammer{c}{k-l}
		}{
			\pas{k-1}!\pas{b+c+k-1}!
		}
\end{multline}
Clearly, by straightforward column operations we can achieve that submatrix
$Q_2^{-1}\cdot Q_1$ is replaced by $\nought$. Expressed as matrix multiplication:
$$
\begin{pmatrix}
\eye & Q_2^{-1}\cdot Q_1 \\
Q_3 & Q4
\end{pmatrix}
\cdot
\begin{pmatrix}
\eye & -Q_2^{-1}\cdot Q_1 \\
\nought & \eye
\end{pmatrix}
=
\begin{pmatrix}
\eye & \nought \\
Q_3 & F
\end{pmatrix},
$$
where $F$ is the product of matrices
$$
F =
\begin{pmatrix}
Q_3 & Q_4 \\
\end{pmatrix}
\cdot
\begin{pmatrix}
-Q_2^{-1}\cdot Q_1 \\
\eye
\end{pmatrix}
= Q_4 - Q_3\cdot Q_2^{-1}\cdot Q_1.
$$
Combining \eqref{eq:Q2^-1Q1ij} and \eqref{eq:Q3ij}, we see that the $\pas{i,j}$--entry of
$Q_3\cdot Q_2^{-1}\cdot Q_1$ is the triple sum
{\small
\begin{multline}
\label{eq:Q3Q2^-1Q1ij}
\pas{Q_3\cdot Q_2^{-1}\cdot Q_1}_{i,j} =
\sum_{t=1}^a \binom{b+c-2i+1}{c-i+t-p} \\
\sum_{l=1}^a
	\pas{-1}^{t+l}\pas{b+l-1}!\pas{c+t-1}!\\
	\binom{2j-1}{l+j-p-1}
    \sum_{k=1}^a
    	\binom{k-1}{t-1}\binom{k-1}{l-1}
		\frac{
			\pochhammer{b}{k-t}\pochhammer{c}{k-l}
		}{
			\pas{k-1}!\pas{b+c+k-1}!
		}.
\end{multline}
}
(Note that the $\pas{i,j}$--entry \eqref{eq:F} \EM{does not depend on $d$}.)

So by combining this with \eqref{eq:Q3ij}, we deduce: 
\begin{cor}
\label{cor:F}
Let $Q_1$, $Q_2$, $Q_3$ and $Q_4$ be the submatrices of the matrix $Q$ underlying
the determinant giving $\evencount{a,b,c,d,p}$. Then $\evencount{a,b,c,d,p}$ (and thus
the number of tilings  of the $\pas{a,b,c}$--hexagon with an (even) intrusion
of length $d$ at position $p$) is given as
\begin{equation}
\evencount{a,b,c,d,p} = \det F\cdot\macmahon{a,b,c} = \det F\cdot\prod_{i=0}^{a-1} i! \frac{\pas{b+c+i}!}{\pas{b+i}! \pas{c+i}!},
\end{equation}
where $F = Q_4 - \pas{Q_3\cdot Q_2^{-1}\cdot Q_1}$.
\end{cor}
\secC{Special case $d=1$, once again.}
Note that $F$ is the $\pas{d\times d}$--matrix with $\pas{i,j}$--entry
\begin{equation}
\label{eq:F}
F_{i,j} = 
\binom{2\pas{j-i}}{j-i} - 
\pas{Q_3\cdot Q_2^{-1}\cdot Q_1}_{i,j}, 
\end{equation}
where $\pas{Q_3\cdot Q_2^{-1}\cdot Q_1}_{i,j}$ is given by \eqref{eq:Q3Q2^-1Q1ij}, so
for the special case $d=1$, the determinant of matrix $F$ is simply $F_{1,1}$.
Combining this
with our result for $\evencount{a,b,c,1,p}$ (i.e., for the special case $d=1$; see
equation \eqref{eq:d=1} in Proposition~\ref{pro:d=1})
gives (after straightforward cancellations and simplifications; observe that the
sum over $l$ only contributes two non--zero summands)
the following summation formula:
\begin{multline}
\label{eq:sum-formula}
\left(-1\right)^{p} \left(b + p - 1\right)! \sum_{t=1}^{a} \left(-1\right)^{t} {\binom{b + c - 1}{c - p + t - 1}} \left(c + t - 1\right)! \\
\sum_{k=1}^{a} \frac{\left(- \left(b + p\right) {\binom{k - 1}{p}} + \left(c + k - p - 1\right) {\binom{k - 1}{p - 1}}\right) \pochhammer{b}{k - t} \pochhammer{c}{k - p - 1} {\binom{k - 1}{t - 1}}}{\left(k - 1\right)! \left(b + c + k - 1\right)!} \\
= 1-\binom{a}{a-p}\frac{\pochhammer{b}{p}\pochhammer{c}{a-p}}{\pochhammer{b+c}{a}}.
\end{multline}
For the special case $p=0$, \eqref{eq:sum-formula} reads (after some simplification)
\begin{equation}
b! \sum_{t=0}^{a - 1} \left(-1\right)^{t} \pochhammer{b - t}{c + t}
\sum_{k=0}^{a - 1} \frac{\pochhammer{b}{k - t} {\binom{k}{t}} {\binom{c + k - 1}{k}}}{\left(b + c + k\right)!}
= 1-\frac{\pochhammer{c}{a}}{\pochhammer{b+c}{a}}.
\end{equation}
For $p>0$, we may rewrite \eqref{eq:sum-formula} as
\begin{multline}
\left(-1\right)^{p} \left(b + p - 1\right)! \sum_{t=1}^{a} \left(-1\right)^{t} {\binom{b + c - 1}{c - p + t - 1}} \left(c + t - 1\right)! \\
\sum_{k=1}^{a} \frac{\left(- \frac{b k}{p} + b + c - 1\right) \pochhammer{b}{k - t} \pochhammer{c}{k - p - 1} {\binom{k - 1}{p - 1}} {\binom{k - 1}{t - 1}}}{\left(k - 1\right)! \left(b + c + k - 1\right)!} \\
= 1-\binom{a}{a-p}\frac{\pochhammer{b}{p}\pochhammer{c}{a-p}}{\pochhammer{b+c}{a}}.
\end{multline}

As a direct consequence of Byun's formula \eqref{eq:byun-even}, we obtain:
\begin{pro}
If we set $a=2p$ in Corollary~\ref{cor:F}, then the determinant factors nicely:
$$
\det\left.F\right\vert_{a=2p} = 4^{d p} \prod_{k=1}^d
\frac{ \left(k-\frac{1}{2}\right)_p (b-k+1)_p (c-k+1)_p}{(k)_p (b+c-2 k+2)_{2 p}}.
$$
\end{pro}



\begin{thebibliography}{10}

\bibitem{andre:1887}
D{\'e}sir{\'e} Andr{\'e}.
\newblock Solution directe du probleme r{\'e}solu par {M}. {B}ertrand.
\newblock {\em CR Acad. Sci. Paris}, 105(436):7, 1887.

\bibitem{bressoud:1999}
D.~Bressoud.
\newblock {\em Proofs and Confirmations: The Story of the Alternating Sign
  Matrix Conjecture}.
\newblock Cambridge University Press, New York, 1999.

\bibitem{Byun:2022:LTOAHWAHI}
Seok~Hyun Byun.
\newblock Lozenge tilings of a hexagon with a horizontal intrusion.
\newblock {\em Annals of Combinatorics}, page 28 pages, 2022.

\bibitem{Ciucu:1997:EOPMIGWRS}
M.~Ciucu.
\newblock Enumeration of perfect matchings in graphs with reflective symmetry.
\newblock {\em J.\ Combin.\ Theory Ser.\ 1}, 77:67--97, 1997.

\bibitem{Ciucu:2005:AGOMF}
M.~Ciucu.
\newblock Plane {P}artition {I}: A generalization of {M}ac{M}ahon's formula.
\newblock {\em Mem. Amer. Math. Soc.}, 178(839):107--144, 2005.

\bibitem{dodgson:1866}
C.L. Dodgson.
\newblock Condensation of determinants, being a new and brief method for
  computing their arithmetic values.
\newblock {\em Proceed. Roy. Soc. London}, 15:150--155, 1866.

\bibitem{Gessel-Viennot:1998:DPAPP}
I.M. Gessel and X.~Viennot.
\newblock Determinants, paths, and plane partitions.
\newblock preprint, 1989.

\bibitem{jacobi:1841a}
C.G. Jacobi.
\newblock De formatione et proprietatibus determinantium.
\newblock {\em Journal f\"ur Reine und Angewandte {M}athematik}, 22:285--318,
  1841.

\bibitem{kasteleyn:1963}
P.W. Kasteleyn.
\newblock Dimer statistics and phase transitions.
\newblock {\em J.~Mathematical Phys.}, 4(2):287--293, 1963.

\bibitem{kuo:2004}
E.~H. Kuo.
\newblock Applications of graphical condensation for enumerating matchings and
  tilings.
\newblock {\em Theoret. Comput. Sci.}, 319:29--57, 2004.

\bibitem{Lindstroem:1973:OTVROIM}
B.~Lindstr\"om.
\newblock On the vector representation of induced matroids.
\newblock {\em Bull. London Math. Soc.}, 5:85--90, 1973.

\bibitem{MacMahon:1916:CA2}
P.~A. MacMahon.
\newblock {\em Combinatory Analysis}, volume~2.
\newblock Cambridge University Press, 1916.

\bibitem{muir:1906}
T.~Muir.
\newblock {\em The Theory of Determinants in the historical order of
  development}.
\newblock MacMillan and Co., Limited, 1906.

\bibitem{zb:risc}
Peter Paule, Markus Schorn, and Axel Riese.
\newblock Fast {Z}eilberger package version 3.61.
\newblock Technical report, Research Institute for Symbolic Computation
  (RISC),, 2022.

\bibitem{percus:1969}
J.~K. Percus.
\newblock One more technique for the dimer problem.
\newblock {\em J.~Mathematical Phys.}, 10(10):1881--1884, 1969.

\bibitem{Zeilberger:1991:TMOCT}
Doron Zeilberger.
\newblock The method of creative telescoping.
\newblock {\em J. Symb. Comput.}, 11:195--204, 1991.

\end{thebibliography}
\end{document}